\documentclass[11pt]{article}
\usepackage[left=2cm, top=2.8cm, right=2cm, bottom=2.8cm, bindingoffset=0.5cm]{geometry}
\usepackage[giveninits=true, maxbibnames=99, doi=false, isbn=false, url=false, eprint=false]{biblatex}
\renewbibmacro*{volume+number+eid}{%
  \printfield{volume}%
  \setunit*{\addnbthinspace}% NEW (optional); there's also \addnbthinspace
  \printfield{number}%
  \setunit{\addcomma\space}%
  \printfield{eid}}
\DeclareFieldFormat[article]{number}{\mkbibparens{#1}}

\usepackage{graphicx} % Required for inserting images
\usepackage{enumitem}
\usepackage{amssymb}
\usepackage{amsmath}
\usepackage{amsfonts}
\usepackage{amsthm}
\usepackage{stmaryrd}
\usepackage[ruled,vlined]{algorithm2e}
\usepackage{hyperref}
\hypersetup{
  colorlinks,
  citecolor = red,
  linkcolor = blue,
  urlcolor = blue
}
\usepackage{subcaption}
\allowdisplaybreaks

\usepackage{tikz}
\usetikzlibrary{patterns}
\usetikzlibrary{shapes.geometric}
\usetikzlibrary{arrows.meta}

\newcommand{\R}{\mathbb{R}}
\newcommand{\N}{\mathbb{N}}
\newcommand{\cH}{\mathcal{H}}

\DeclareMathOperator{\prox}{prox}

\DeclareMathOperator{\dom}{dom}

\DeclareMathOperator{\dist}{dist}

\DeclareMathOperator*{\argmin}{\arg\!\min}

\newtheorem{prop}{Proposition}[section]
\newtheorem*{prop*}{Proposition}
\newtheorem{rmk}[prop]{Remark}
\newtheorem*{rmk*}{Remark}
\newtheorem{thm}[prop]{Theorem}
\newtheorem*{thm*}{Theorem}

\newtheorem*{defn*}{Definition}
\newtheorem{lem}[prop]{Lemma}
\newtheorem*{lem*}{Lemma}

\newtheorem*{cor*}{Corollary}
\newtheorem{ass}{Assumption}

\makeatletter
\newcommand{\customlabel}[2]{%
	\protected@write \@auxout {}{\string \newlabel {#1}{{#2}{\thepage}{#2}{#1}{}} }%
	\hypertarget{#1}{#2}
}
\makeatother
\newcommand{\labterm}[2]{\customlabel{#1}{\text{$#2$}}}

\usepackage[colorinlistoftodos]{todonotes}

\title{Accelerating Diagonal Methods for Bilevel Optimization: Unified Convergence via Continuous-Time Dynamics}

\author{Radu I. Bo\c t\thanks{Faculty of Mathematics, University of Vienna, Oskar-Morgenstern-Platz 1, 1090 Vienna, Austria.\\ The research of RIB has been supported by the Austrian Science Fund (FWF), projects W 1260 and P 34922-N, respectively. DAH was supported by the Doctoral Programme \emph{Vienna Graduate School on Computational Optimization (VGSCO)}, funded by the Austrian Science Fund (FWF), project W1260-N35. \texttt{email:} \{radu.bot\}\{enis.chenchene\}\{robert.csetnek\}\{david.alexander.hulett\}@univie.ac.at} \and Enis Chenchene\footnotemark[1] \and E. Robert Csetnek\footnotemark[1] \and David A. Hulett\footnotemark[1]}

\date{\today}

\begin{document}

\maketitle

\begin{abstract}
	We analyze fast diagonal methods for simple bilevel programs. Guided by the analysis of the corresponding continuous-time dynamics, we provide a unified convergence analysis under general geometric conditions, including H{\"o}lderian growth and the Attouch--Czarnecki condition. Our results yield explicit convergence rates and guarantee weak convergence to a solution of the bilevel problem. In particular, we improve and extend recent results on accelerated schemes, offering novel insights into the trade-offs between geometry, regularization decay, and algorithmic design. Numerical experiments illustrate the advantages of more flexible methods and support our theoretical findings.
\end{abstract}



\section{Introduction}

Bilevel optimization is an increasingly active area in optimization that consists in finding a minimizer of an \emph{outer} function among the minimizers of an \emph{inner} function. It originated in economic game theory in the 50s \cite{dz20}, and found enormous application in modern applied mathematics, such as in hyperparameter optimization \cite{kp13, chb23}, meta-learning \cite{acss19}, and reinforcement learning \cite{bmk20}, see \cite{dz20} and the references therein for a comprehensive account.  

% out porblems
In this paper, we consider a class of \emph{bilevel} optimization problems, where the lower-level problem is independent of the upper-level decision variable, within the broader framework of \emph{hierarchical optimization} that read as:
\begin{equation}\label{eq:bilevel_problem}
        \begin{aligned}
        \min \ & H(x) := h(x) + \hat h(x),\\
        \text{subject to} \ & x \in \argmin_{z \in \cH} \ F(z) := f(z) + \hat f(z),
        \end{aligned}
\end{equation}
where $\cH$ is a real Hilbert space, $\hat h, \hat f : \cH \to \R \cup \{+\infty\}$ are proper, convex and lower semicontinuous functions, and $h, f :  \cH \to \R$ are convex and Fr\'echet differentiable functions with $L_{\nabla h}$- and $L_{\nabla f}$- Lipschitz continuous gradients, respectively. We further assume that $\argmin F$, the set of global minimizers of $F$, is nonempty and denote by $\min F$ the infimal value of $F$. Moreover, we assume that $\inf_{x \in \cH} H(x)>-\infty$, that the set of solutions to \eqref{eq:bilevel_problem} is nonempty, and denote its optimal value by $\min_{\argmin F} H$.

A classical approach to tackle \eqref{eq:bilevel_problem} is by an approximation argument, the so-called \emph{viscosity method} \cite{attouch96, ab93}. This method consists of replacing \eqref{eq:bilevel_problem} with the regularized optimization problem
\begin{equation}\label{eq:regularized}
	\min_{x \in \cH} \ F(x) + \varepsilon H(x), \quad \text{for} \ \varepsilon >0,
\end{equation} 
and recovering a solution to \eqref{eq:bilevel_problem} by considering (weak) cluster points of minimizers of \eqref{eq:regularized} as $\varepsilon\to 0$. This classical approach appears across a wide range of mathematical disciplines, including the calculus of variations \cite{et74}, numerical linear algebra \cite{cs94}, optimal transport \cite{cuturi13}, control theory and partial differential equations \cite{cl83}, among others. Its origins trace back to the work of Tikhonov \cite{tichonov63} in the context of ill-posed linear problems, where the regularization term takes the form $H:=\|\cdot\|^2$. While in the so-called \emph{exact penalization} setting \cite{ft08}, solving \eqref{eq:regularized} for $\varepsilon$ small enough actually yields solutions to \eqref{eq:bilevel_problem}, such a threshold is usually not known a-priori. Thus, successive solutions of \eqref{eq:regularized} with $\varepsilon\equiv \varepsilon_k$ and $\varepsilon_k\to 0$ as $k\to +\infty$ are needed to approach a solution to \eqref{eq:bilevel_problem}. Quite surprisingly, it turns out that we can replace these successive minimization problems by a \emph{single} iteration of a minimization algorithm. This defines the so-called \emph{diagonal} or \emph{Tikhonov} methods.

% diagonal approach
The first breakthrough in this direction was proposed by Cabot in \cite{cabot05}, who replaced the minimization oracle \eqref{eq:regularized} with a proximal step on \eqref{eq:regularized} and highlighted the role of a slow-vanishing behavior of $(\varepsilon_k)_{k \in \N}$. Later, Solodov provided the first gradient-type algorithm \cite{solodov07} designed for solving \eqref{eq:bilevel_problem} with $\hat h=0$ and $\hat f=\iota_D$, where $D \subseteq \cH$ is a nonempty, convex and closed set and $\iota_D$ denotes its \emph{indicator function}. In the smooth setting, i.e., $\hat h= \hat f = 0$, these methods can be understood as a discretization of the dynamical system
\begin{equation}\label{eq:intro_tichonov_flow}
	\dot x(t) + \nabla f(x(t)) + \varepsilon(t) \nabla h(x(t)) = 0, \quad \text{for} \ t \geq t_0,
\end{equation}
where $t_0 \geq 0$ and $\varepsilon:[t_0, +\infty)\to \R_+$ is a \emph{vanishing regularization function}. Observe that indeed \eqref{eq:intro_tichonov_flow} can be seen as a gradient flow applied to \eqref{eq:regularized} with a vanishing parameter function $t \mapsto \varepsilon(t)$. The convergence of \eqref{eq:intro_tichonov_flow} has been first established by Attouch and Czarnecki in \cite{ac10}, who showed that, under a certain growth condition on $f$ (Assumption \ref{ass:attouch_czarnecki}), $x(t)$ converges weakly to an optimal solution of \eqref{eq:bilevel_problem} with a convergence rate for $f(x(t)) - \min f$ expressed in terms of $\varepsilon(t)$, cf.~\cite[Theorem 5.1]{ac10}. More recently, by replacing $\nabla f$ with a  general monotone and Lipschitz continuous operators and setting $h:=\tfrac{1}{2}\|\cdot\|^2$, \eqref{eq:intro_tichonov_flow} and some of its discretizations have gained renewed interest for their ability to yield strong convergence toward the minimal norm solution, as well as their connection to Halpern-type methods \cite{ss17, lieder21, yr21}, see also \cite{bc24, bcf24}.

% the case h non strongly convex
Building on the work of Cabot and Solodov, Peypouquet proposed in \cite{peypouquet12} an adaptation of the Attouch--Czarnecki condition \cite{ac10} for gradient systems, establishing weak convergence of the generated sequence to a solution of \eqref{eq:bilevel_problem}. In finite-dimensional settings, convergence results can still be obtained under milder assumptions, e.g., coercivity of $H$, even in the general framework of \eqref{eq:bilevel_problem}. For example, \cite{ms23}, which complements and extends \cite{ky21}, highlights an inherent trade-off between inner and outer function performance, similar to \eqref{eq:discussion_rates_3}. Recently, backtracking strategies have been developed for functions with merely locally Lipschitz continuous gradients \cite{ltvp25}. Additional notable extensions---without claiming to be exhaustive---include bundle methods \cite{solodov07_bundle}, cutting plane approaches \cite{cjyhm24}, and algorithms tailored for variational inequality inner problems \cite{ky21}.

% rates
Building on this rationale, Merchav, Sabach, and Teboulle \cite{mst24} were the first to observe that applying a single iteration of Nesterov’s fast gradient method---or FISTA in the composite (smooth $+$ non-smooth) setting---to \eqref{eq:regularized} can accelerate convergence, while still generating sequences that converge weakly to solutions of \eqref{eq:bilevel_problem}. In continuous time, these methods can be interpreted as discretizations of the dynamical system \cite{sbc16}
\begin{equation}\label{eq:intro_su_boyd_candes_tichonov}
	\ddot x(t) + \frac{\alpha}{t}\dot x(t) + \nabla f(x(t)) + \varepsilon(t) \nabla h(x(t)) = 0, \quad \text{for} \ t \geq t_0,
\end{equation}
where $t_0 \geq 0$ and $\alpha>3$. This system is known to provide a theoretical foundation for analyzing corresponding discrete numerical algorithms, including those involving non-smooth proximal terms. Building on the extensive body of work developed over the past decade on second-order continuous and discrete-time dynamical systems---greatly shaped by Attouch's influential contributions---and guided by the analysis of the continuous-time systems \eqref{eq:intro_tichonov_flow} and \eqref{eq:intro_su_boyd_candes_tichonov}, we propose two new flexible and fast numerical algorithms. Our approach allows us to i) devise more general numerical schemes than in \cite{mst24, ltvp25}, see Algorithms \ref{alg:first_order} and \ref{alg:second_order}, and ii) significantly improve several convergence results therein, yielding, among other results, little-$o$ rates and weak convergence of trajectories in a Hilbert space setting, see Section \ref{sec:contribution} for a complete overview. In doing so, we unveil an intriguing one-to-one correspondence on the convergence behavior of \eqref{eq:intro_tichonov_flow} and \eqref{eq:intro_su_boyd_candes_tichonov} and the corresponding numerical algorithms, which shows that the techniques for their analysis are fundamentally equivalent.

\subsection{Contribution}\label{sec:contribution}

We investigate the asymptotic behavior of two bilevel optimization algorithms: a proximal-gradient method (Algorithm~\ref{alg:first_order}) and its accelerated variant with Nesterov momentum (Algorithm~\ref{alg:second_order}). These algorithms are derived as discrete-time counterparts of the dynamical systems \eqref{eq:intro_tichonov_flow} and \eqref{eq:intro_su_boyd_candes_tichonov}, respectively. For simplicity, we restrict ourselves to regularization parameters of the form
\begin{equation}\label{eq:def_epsilon}
	\varepsilon_k := \frac{c}{(k + \beta )^\delta}, \quad \text{for $\beta, c, \delta > 0$ and $k\geq 0$}.
\end{equation}
Additionally, we focus on a broad geometric setting in which the inner function has either a $\rho$-H{\"o}lderian growth (Assumption \ref{ass:holderian_growth}) or satisfies the weaker Attouch--Czarnecki condition (Assumption \ref{ass:attouch_czarnecki}), a setting that offers a good trade-off between geometric assumptions and convergence guarantees. We emphasize the critical role of the interplay between the geometry of $F$ and the vanishing behavior of $(\varepsilon_k)_{k \in \N}$ (see Figure~\ref{fig:geometric_setting}), and its impact on the rates of convergence. We defer results under milder assumptions, aligned to those in \cite{ms23, ltvp25}, to a miscellaneous Section \ref{sec:discussion}.

Our main results, Theorem \ref{thm: first order algorithm} and Theorem \ref{thm:second_order_main}, can be summarized as follows: Let $(x_{k})_{k\geq 0}$ be the sequence generated by Algorithm \ref{alg:first_order} (in which case we set $\eta = 1$) or Algorithm \ref{alg:second_order} (in which case we set $\eta=2$). Then:
	\begin{itemize}
		\item Case $\delta > \eta$: It holds $F(x_k) - \min F= \mathcal{O}(k^{-1})$ and $F(x_k) - \min F = \mathcal{O}(k^{-2})$ as $k\to +\infty$ for Algorithm \ref{alg:first_order} and \ref{alg:second_order}, respectively. However, in either case, there is no guarantee that $(x_k)_{k\geq 0}$ converges to a solution to \eqref{eq:bilevel_problem} and, indeed, simple numerical experiments show that this is not the case in general.
		\item Case $\delta = \eta$: If $F$ satisfies the Attouch--Czarnecki condition, i.e., Assumption \ref{ass:attouch_czarnecki} with respect to $\eta$: 
		\begin{equation}
		\begin{aligned}
			\text{Algorithm \ref{alg:first_order}}: && F(x_{k}) - \min F = o\big(k^{-1}\big), && \text{Algorithm \ref{alg:second_order}}: && F(x_{k}) - \min F = o\big(k^{-2}\big).
		\end{aligned}
		\end{equation}
		Furthermore, in both cases $H(x_{k}) \to H(x^{*})$ as $k\to +\infty$, and $(x_{k})_{k\geq 0}$ converges weakly towards a solution to the bilevel problem \eqref{eq:bilevel_problem}.
		\item Case $\frac{\eta}{\rho^\star} < \delta < \eta$: If $F$ satisfies Assumption \ref{ass:holderian_growth} with exponent $\rho$, then
		\begin{equation}\label{eq:intro_trade_off}
			\begin{aligned}
			&\text{Algorithm \ref{alg:first_order}}: && F(x_{k}) - \min F = o\big(k^{-1}\big), \quad \text{and} \quad |H(x_{k}) - \min\nolimits_{\argmin F} H| = o\big(k^{\delta - 1}\big),\\
			&\text{Algorithm \ref{alg:second_order}}: && F(x_{k}) - \min F = o\big(k^{-2}\big), \quad \text{and} \quad |H(x_{k}) - \min\nolimits_{\argmin F} H| = o\big(k^{\delta - 2}\big),
			\end{aligned}
		\end{equation}
		as $k\to +\infty$. Furthermore, in both cases $(x_{k})_{k\geq 0}$ converges weakly towards a solution to the bilevel problem \eqref{eq:bilevel_problem}.
	\end{itemize}
In the smooth case, i.e., $\hat h= \hat f = 0$, all our results admit continuous-time counterparts for the dynamical systems \eqref{eq:intro_tichonov_flow} and \eqref{eq:intro_su_boyd_candes_tichonov}, which are of independent interest and are discussed in Appendix \ref{sec:continuous_time}.

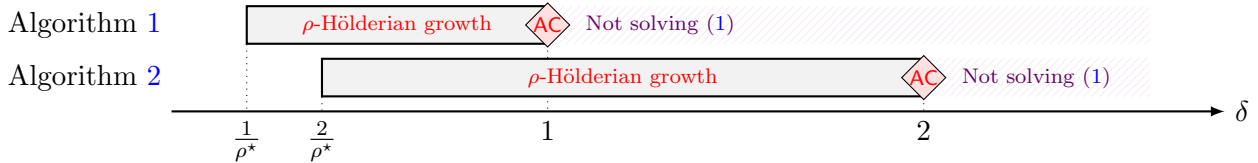
\begin{figure}[t]
	\begin{tikzpicture}[>=latex]
		
		% Axis
		\draw[->, thick] (0,0.6) -- (14,0.6) node[right] {$\delta$};
		
		% Vertical markers
		\draw[dotted] (1, 2) -- (1,0.6) node[below] {$ \tfrac{1}{\rho^\star}$};
		\draw[dotted] (2, 1.3) -- (2,0.6) node[below] {$\tfrac{2}{\rho^\star}$};
		\draw[dotted] (5, 2) -- (5,0.6) node[below] {1};
		\draw[dotted] (10,1.3) -- (10,0.6) node[below] {2};
		
		% Algorithm labels
		\node[anchor=east] at (0,1.75) {Algorithm \ref{alg:first_order}};
		\node[anchor=east] at (0,1.05) {Algorithm \ref{alg:second_order}};
		
		% Algorithm 1 Box
		\draw[fill=gray!10, thick] (1,2) -- (5,2) -- (5,1.5) -- (1,1.5) -- cycle;
		\node[anchor=center] at (3,1.75) {\scriptsize{\textcolor{red}{$\rho$-H{\"o}lderian growth}}};
		
		% Algorithm 2 Box
		\draw[fill=gray!10, thick] (2,1.3) -- (10,1.3) -- (10,0.8) -- (2,0.8) -- cycle;
		\node[anchor=center] at (6,1.05) {\scriptsize{\textcolor{red}{$\rho$-H{\"o}lderian growth}}};
		
		% Not solving REF - Algorithm 1
		\fill[pattern=north east lines, pattern color=violet!70, opacity=0.3] (5,1.5) rectangle (13,2);
		\node at (6.5,1.75) {\scriptsize \textcolor{violet!80!black}{Not solving \eqref{eq:bilevel_problem}}};
		
		% Not solving REF - Algorithm 2
		\fill[pattern=north east lines, pattern color=violet!70, opacity=0.3] (10, 0.8) rectangle (13,1.3);
		\node at (11.5,1.05) {\scriptsize \textcolor{violet!80!black}{Not solving \eqref{eq:bilevel_problem}}};
		
		% AC diamonds
		\node[diamond, draw, text=red, fill=pink!50, minimum width=0.1cm, minimum height = 0.3cm, inner sep=0.1pt] (d) at (5,1.75) {\scriptsize{\textsf{AC}}};
		
		\node[diamond, draw, text=red, fill=pink!50, minimum width=0.1cm, minimum height = 0.3cm, inner sep=0.1pt] (d) at (10,1.05) {\scriptsize{\textsf{AC}}};
	\end{tikzpicture}
\caption{Geometric setting considered in this paper, where $\rho^*$ is the dual exponent of $\rho$, see \eqref{eq:def_rho_star}. For a discussion of results obtained under even milder assumptions, see Section~\ref{sec:discussion}.}
\label{fig:geometric_setting}
\end{figure}

\subsection{Relation to previous work}

Accompanied with Section \ref{sec:discussion}, our comprehensive analysis complements and improves the recent works \cite{ms23, ltvp25, mst24}. First, all results are formulated in a infinite dimensional Hilbert space setting, in line with \cite{peypouquet12, ac10}, where the case $\eta=1$ and $\hat h =\hat f=0$ is analyzed with a general regularization sequence/parameter. However, only weak convergence to an optimal solution of \eqref{eq:bilevel_problem} is provided (in both continuous and discrete time), along with a convergence rate for the inner function (in continuous time only).

In the case $\eta = 1$, that is, for Algorithm~\ref{alg:first_order}---which coincides with the proximal-gradient method introduced in \cite{ltvp25}---we provide a convergence analysis by choosing the regularization sequence $(\varepsilon_k)_{k \in \N}$ as in \eqref{eq:def_epsilon} and making the geometric assumptions of Figure \ref{fig:geometric_setting}, but allowing larger step sizes. As discussed in Section \ref{sec:discussion}, our analysis can be extended to accommodate milder assumptions---such as the coercivity of $F$ or $H$---and more general parameter choices. However, without stronger assumptions, convergence rates are typically only available in the \emph{ergodic} or \emph{best-iterate} sense, and convergence to solutions to \eqref{eq:bilevel_problem} can generally only be characterized by the vanishing distance to the solution set of \eqref{eq:bilevel_problem}, and stated in finite-dimensional settings; see Section \ref{sec:discussion} for further details. We obtain best-iterate results similar to \cite[Theorem 3]{ms23}, but assuming coercivity on either function $F$ or $H$ instead of $H$, and for the method in Algorithm \ref{alg:first_order}, instead of the iterative scheme in \cite{ms23}. Furthermore, our analysis can be extended to accommodate general parameter sequences (see Section \ref{sec:discussion}). However, we find little motivation to do so, and numerical results---e.g. those reported in \cite{bc24}---seem to indicate that polynomial choices tend to give the best performance in Tikhonov methods.

In the case $\eta=2$, that is, for Algorithm \ref{alg:second_order}, the only method available for a fair comparison is that of  \cite{mst24}. Our results substantially extend those in \cite{mst24} in several directions. First, we introduce a more general algorithmic framework. Second, we allow for $\rho$-H{\"o}lderian growth with any $\rho\in (1, 2]$, rather than restricting to the quadratic case $\rho=2$. Third, our analysis is carried out in an infinite-dimensional Hilbert space setting. Fourth, for $\delta \in (\frac{2}{\rho^\star}, 2)$, we refine the convergence rates established in \cite{mst24}, improving the simultaneous bounds from $\mathcal{O}(k^{-2})$ and $\mathcal{O}(k^{\delta - 2})$ to  $o(k^{-2})$ and $o(k^{\delta - 2})$ as $k\to+\infty$, respectively, by building on techniques from \cite{acfr23}. In addition, we complement the results of \cite{mst24} by treating the case $\delta=2$ under the Attouch--Czarnecki assumption, which is weaker than the H{\"o}lderian growth condition, and by establishing weak convergence of the sequence$(x_k)_{k \in \N}$ in this setting. Finally, our analysis departs fundamentally from prior work in that all results are derived from the dissipation of a single energy functional---the so-called \emph{Lyapunov function}.

We also emphasize that each of our results admits a continuous-time counterpart, presented in Appendix \ref{sec:continuous_time}, which in fact guided the convergence analysis of the corresponding numerical methods.

\subsection{Preliminaries and notation}\label{sec:preliminaries}

For a proper function $f: \cH \to \R\cup \{+\infty\}$ we denote by $\dom f := \{x \in \cH \ : \ f(x) < +\infty\}$ its \emph{domain}, by $\min f$ its \emph{infimal value}, and by $\argmin f := \{x \in \cH \ : \ f(x) = \min f\}$ its \emph{set of minimizers}. For a proper, convex and lower semicontinuous function $f: \cH \to \R\cup \{+\infty\}$ we denote by $f^* : \cH \to \R\cup \{+\infty\}$  its \emph{Fenchel--Moreau conjugate}, defined by $f^{*}(z) := \sup_{x\in \cH} \{ \langle x, z\rangle - f(x)\}$. The \emph{proximity operator} of $f$ with modulus $s >0$ is defined for all $x \in \cH$ by 
\begin{equation*}
    \prox_{s f}(x):= \argmin_{y \in \cH} \Big\{ f(y) + \frac{1}{2s}\|y - x\|^2\Big\}.
\end{equation*}
If  $f: \cH \to \R$ is convex and Fréchet differentiable with a $L_{\nabla f}$-Lipschitz continuous gradient, we say that $f$ is \emph{$L$-smooth}. We denote by $\iota_C$ the indicator function associated to a nonempty, convex and closed set $C \subseteq \cH$, by $\sigma_C:=\iota_C^*$ the support function of the set $C$, and by $N_C:=\partial \iota_C$ its normal cone.

In our geometric setting, we assume Hölderian error bound for the inner function, and a weaker summability condition due to Attouch and Czarnecki in \cite{ac10}. We will now present both assumptions. 
\begin{ass}[H{\"o}lderian growth condition]\label{ass:holderian_growth}
     A proper, convex and lower semicontinuous function $f: \cH \to \R\cup \{+\infty\}$ with $\argmin f \neq \emptyset$ is said to satisfy a \emph{H{\"o}lderian growth condition with exponent $\rho\in (1, 2]$} if for some $\tau > 0$ we have
    \begin{equation}\label{eq: Hölderian error bound for f}
        \tau \rho^{-1} \dist(x, \argmin f)^{\rho} \leq f(x) - \min f \quad \text{for all} \ x\in \mathcal{H}.
    \end{equation}
\end{ass}
\begin{ass}[Attouch and Czarnecki, \cite{ac10}]\label{ass:attouch_czarnecki}
    A proper, convex and lower semicontinuous function $f: \cH \to \R\cup \{+\infty\}$ with $\argmin f \neq \emptyset$ is said to satisfy the \emph{Attouch--Czarnecki condition} with respect to $\eta\in \{1, 2\}$ if for any solution $x^{*}$ to the bilevel problem \eqref{eq:bilevel_problem}, for all $p \in N_{\argmin f}(x^{*})$, and $(\varepsilon_k)_{k\geq 0}$ defined as in \eqref{eq:def_epsilon} we have that
	\begin{equation}\label{eq:acp_assumption discrete}
		\sum_{k = 0}^{+\infty} k^{\eta - 1} \bigg[ (f - \min f)^{*}\left( \varepsilon_k p \right) - \sigma_{\argmin f}\left( \varepsilon_k p \right)\bigg] < +\infty.
	\end{equation}
\end{ass}

\noindent The Attouch--Czarnecki condition is known to be weaker than the H{\"o}lderian growth condition, since, whenever $f$ satisfies Assumption \ref{ass:holderian_growth},
\begin{equation}\label{eq: bound for conjugates}
	0\leq (f - \min f)^{*}(z) - \sigma_{\argmin f}(z) \leq \tau^{1-\rho^\star}\frac{1}{\rho^\star} \| z\|^{\rho^\star} \quad \text{for all} \ z\in\mathcal{H},
\end{equation}
where the first inequality is a consequence of the Young--Fenchel inequality, and the second can be derived as in \cite[Section 3.2]{peypouquet12}. Here, the dual conjugate exponent $\rho^\star$ is given by
\begin{equation}\label{eq:def_rho_star}
\frac{1}{\rho} + \frac{1}{\rho^\star} = 1.
\end{equation}
Therefore, as $\varepsilon_k \simeq k^{-\delta}$, Assumption \ref{ass:holderian_growth} implies Assumption \ref{ass:attouch_czarnecki} when $\frac{\eta}{\rho^\star} < \delta$. A further mild technical assumption we will require is a qualification condition.
\begin{ass}[Qualification condition]\label{ass:qualification}
The functions $H$ and $F$ in \eqref{eq:bilevel_problem} are such that
\begin{equation}
	\partial (H + \iota_{\argmin F})(x^*) = \partial H(x^*) + N_{\argmin F}(x^*) \quad \text{for any solution} \ x^* \ \text{to \eqref{eq:bilevel_problem}} .
\end{equation} 
\end{ass}   
\noindent The qualification condition is known to be fulfilled if either there exists $x' \in \dom H \cap \argmin F$ such that $H$ is continuous at $x'$ or if $0 \in \text{int}(\dom H  - \argmin F)$. This assumption proves to be crucial for controlling the (sign-less) residual of \eqref{eq:regularized} with $\varepsilon>0$  in terms of its value at an optimal solution to \eqref{eq:bilevel_problem}. Indeed, picking $p^* \in N_{\argmin F}(x^*)$ such that $-p^* \in \partial H(x^*)$ given by Assumption \ref{ass:qualification}, we get that for all $x \in \cH$
\begin{equation}\label{eq: Attouch-Czarnercki implies the integrability condition of the lemma}
	\begin{aligned}
	 - \Bigl[ F(x) - F(x^{*}) + \varepsilon (H(x) - H(x^{*}))\Bigr] & \leq - (F(x) - \min F) + \varepsilon \langle p^*, x - x^{*}\rangle \\
& = \Bigl[ \langle \varepsilon p^{*}, x\rangle - (F - \min F)(x)\Bigr] - \langle \varepsilon p^{*}, x^{*}\rangle \\ & \leq (F - \min F)^{*} (\varepsilon p^{*}) - \sigma_{\argmin F}(\varepsilon p^{*}),
\end{aligned}
\end{equation}
where we used $\langle \varepsilon p^{*}, x^{*}\rangle=\sigma_{\argmin F}(\varepsilon p^{*})$. Although the Attouch--Czarnecki condition in Assumption \ref{ass:attouch_czarnecki} might appear abstract, \eqref{eq: Attouch-Czarnercki implies the integrability condition of the lemma} shows that it comes naturally in the analysis of Tikhonov-type methods.

\subsection{The link between first and second-order energy estimates}

The analysis of the continuous-time systems \eqref{eq:intro_tichonov_flow} and \eqref{eq:intro_su_boyd_candes_tichonov} presented in Appendix \ref{sec:continuous_time} serves as a foundational guide for conducting the discrete-time Lyapunov analysis of the corresponding numerical schemes---Algorithm \ref{alg:first_order} and Algorithm \ref{alg:second_order}. Inspired by \cite{acfr23}, we introduce a discrete Lyapunov energy sequence $(E_k)_{k\geq 1}$, defined in \eqref{eq:lyapunov_first_order} and \eqref{eq:lyapunov_second_order}, of the form
\begin{equation}\label{eq: general energy function, lemma, discrete}
	E_{k} := t_{k}^{\eta}\left( \Psi_{k-1}(x_k) - \Psi_{k-1}(x^*)\right) + V_{k}, \quad \text{for} \ k\geq 1,
\end{equation}
where $\Psi_{k}(x):=F(x) + \varepsilon_k H(x)$ is the regularized function \eqref{eq:regularized} with regularization parameter  $\varepsilon_k$ defined as in \eqref{eq:def_epsilon}, $V_k\geq 0$ is specified in each case, $x^*$ is a solution to \eqref{eq:bilevel_problem}, $t_k:=\theta(k + \gamma)$ for $\theta>0$ and $\gamma\geq 0$, and $\eta=1$ for Algorithm \ref{alg:first_order}, and $\eta=2$ for Algorithm \ref{alg:second_order}. As in the continuous-time Lyapunov analysis, we arrive at the following dissipativity estimate
\begin{equation}\label{eq: general inequality for E(k+1)-E(k), lemma, discrete}
	E_{k + 1} - E_{k} + \zeta_{k, \delta} (H(x_k) - H(x^*)) + g_{k} \leq -C_{2} k^{\eta - 1} (F(x_k)-F(x^*)), \quad \text{for} \ k \ \text{sufficiently large},
\end{equation}
where $g_{k}\geq 0$ and $\zeta_{k, \delta} \simeq k^{-\delta + \eta - 1}$. The precise form of $g_{k}$ and $\zeta_{k, \delta}$ is not important at this stage. As shown in Lemma \ref{lem: general descent lemma for discrete case}, this Lyapunov inequality underpins much of our analysis, enabling the derivation of key convergence results, including little-$o$ rates and weak convergence of the iterates.

\section{Bilevel Proximal-Gradient method}\label{sec:first_order}

Consider the first-order continuous-time dynamics in \eqref{eq:intro_tichonov_flow}. To handle the non-smoothness inherent in \eqref{eq:bilevel_problem}, we adopt a standard semi-implicit discretization of the monotone inclusion form of \eqref{eq:intro_tichonov_flow}, treating the non-smooth components implicitly and the smooth ones explicitly,\begin{equation}\label{eq:discretization_of_tichonov}
 \frac{x_{k + 1} - x_{k}}{\theta} + \partial \hat f(x_{k+1}) + \nabla f(x_{k}) + \varepsilon_k \partial \hat h(x_{k+1}) + \varepsilon_{k} \nabla h(x_{k}) \ni 0,  \quad \text{for} \ k\geq 0.
\end{equation}
Rearranging suitably and recalling that $\prox_{\tau f}=(I+ \tau \partial f)^{-1}$, we obtain Algorithm \ref{alg:first_order}. To study Algorithm \ref{alg:first_order}, we will closely follow the analysis of \eqref{eq:intro_tichonov_flow}, which is significantly simpler. Let us introduce the following notation:
\begin{equation}\label{eq:notation_first_order}
	\Phi_{k}(x) := f(x) + \varepsilon_{k} h(x), \quad \hat{\Phi}_{k}(x) := \hat{f}(x) + \varepsilon_{k} \hat{h}(x), \quad \Psi_{k}(x) := \Phi_{k}(x) + \hat{\Phi}_{k}(x), \quad \text{for} \ k\geq 0.
\end{equation}
Observe that Algorithm \ref{alg:first_order} can be understood as a standard proximal-gradient algorithm for minimizing the sum of the non-smooth function $\hat \Phi_k$ and the smooth function $\Phi_k$. Therefore, the reader can expect that the step size $\theta$ will be picked to satisfy $\theta \in (0, 2/L_k)$, where $L_k:=L_{\nabla f} + \varepsilon_k L_{\nabla h}$. Since $\varepsilon_k\to 0$, any choice $\theta \in (0, 2/L_{\nabla f})$ will suffice to show the convergence of the algorithm.

\begin{algorithm}[t]
    \DontPrintSemicolon
    \KwIn{$x_{0} \in\mathcal{H}$, $0 < \theta < \frac{2}{L_{\nabla 
 f}}$ \hfill \textcolor{gray}{$\rhd$ $\varepsilon_k$ as in \eqref{eq:def_epsilon}} }
        \For{$k = 1, 2, \ldots$}{
            $x_{k + 1} := \prox_{\theta (\hat f + \varepsilon_k \hat h)} \big( x_{k} - \theta (\nabla f(x_{k}) + \varepsilon_k \nabla h(x_k) ) \big)$\;
        }
    \caption{Bilevel Proximal-Gradient method\label{Bilevel Proximal Gradient algorithm}}
    \label{alg:first_order}
\end{algorithm}

For analyzing the asymptotic behavior of Algorithm \ref{alg:first_order}, we make use of the following discrete Lyapunov energy function 
\begin{equation}\label{eq:lyapunov_first_order}
    E_{k}^{\lambda} := t_{k} (\Psi_{k - 1} (x_{k}) - \Psi_{k - 1}(x^{*})) + \frac{\lambda}{2} \| x_{k} - x^{*}\|^{2}, \quad \text{for} \ k \geq 1,
\end{equation}
where $t_k := \theta(k + \gamma)$, with $\gamma \geq 0$, represents the discrete time, $x^*$ is a solution to \eqref{eq:bilevel_problem}, and $\lambda > 1$. We thus obtain the following statement.
\begin{lem}\label{lem:dissipation_first_order}
	Let $(x_k)_{k \geq 0}$ be the sequence generated by Algorithm \ref{alg:first_order} and $(E_k^\lambda)_{k \geq 1}$ be defined as in \eqref{eq:lyapunov_first_order}. Then, there exists $k_0 \geq 1$ such that for all $k \geq k_0$
	\begin{equation}\label{eq:dissipation_first_order}
		E_{k + 1}^{\lambda} - E_{k}^{\lambda} + \zeta_{k, \delta} (H(x_{k}) - H(x^{*})) \leq - \theta (\lambda - 1) (F(x_{k}) - F(x^{*})), 
	\end{equation}
	where $(\zeta_{k, \delta})_{k \geq 1}$ is a sequence defined by
	\[
	\zeta_{k, \delta} :=  \theta(\lambda - 1)\varepsilon_{k} -  t_{k} (\varepsilon_{k} - \varepsilon_{k - 1}).
	\]
\end{lem}
\begin{proof}
	Let $k \geq 1$. We have 
	\begin{align}
		E_{k + 1}^{\lambda} - E_{k}^{\lambda} = &\: t_{k + 1} (\Psi_{k}(x_{k + 1}) - \Psi_{k}(x^{*})) - t_{k} (\Psi_{k}(x_{k}) - \Psi_{k}(x^{*})) \label{eq: E(k+1)-E(k), Bilevel ProxGradient, line 1}\\
		&+ t_{k} (\Psi_{k} (x_{k}) - \Psi_{k}(x^{*})) - t_{k} (\Psi_{k - 1}(x_{k}) - \Psi_{k - 1}(x^{*})) \label{eq: E(k+1)-E(k), Bilevel ProxGradient, line 2}\\
		&+ \frac{\lambda}{2} \| x_{k + 1} - x^{*}\|^{2} - \frac{\lambda}{2} \| x_{k} - x^{*}\|^{2}. \label{eq: E(k+1)-E(k), Bilevel ProxGradient, line 3}
	\end{align}
	We work with each line separately. For line \eqref{eq: E(k+1)-E(k), Bilevel ProxGradient, line 1}, we use Lemma \ref{lem: proximal step}. The proximal step in Algorithm \ref{Bilevel Proximal Gradient algorithm} can be written as 
	\begin{equation}\label{eq: x(k+1) expressed in terms of G(xk), Bilevel ProxGradient}
		x_{k + 1} = x_{k} - \theta G_{k} (x_{k}), \quad \text{where} \quad G_{k}(y) := \frac{1}{\theta} \Bigl[ y - \prox_{\theta \hat{\Phi}_{k}} (y - \theta \nabla \Phi_{k}(y))\Bigr].
	\end{equation}
	Since $\Phi_{k}$ is $L_{k}$-smooth, according to Lemma \ref{lem: proximal step}, for all $x, y\in\mathcal{H}$ it holds
	\begin{equation*}
		\Psi_{k}(y - \theta G_{k}(y))\leq  \Psi_{k}(x) - \frac{\theta}{2} (2 - \theta L_{k}) \| G_{k}(y)\|^{2} + \langle G_{k}(y), y - x\rangle. 
	\end{equation*}
	We use this inequality twice. First, set $y = x := x_{k}$. We obtain
	\begin{equation}\label{eq: E(k+1)-E(k), Bilevel ProxGradient, line 1a}
		\Psi_{k}(x_{k + 1}) \leq \Psi_{k}(x_{k}) - \frac{\theta}{2} (2 - \theta L_{k}) \| G_{k}(x_{k})\|^{2}. 
	\end{equation}
	Now, set $y := x_{k}$ and $x := x^{*}$. This yields 
	\begin{equation}\label{eq: E(k+1)-E(k), Bilevel ProxGradient, line 1b}
		\Psi_{k}(x_{k + 1}) \leq \Psi_{k}(x^{*}) - \frac{\theta}{2}(2 - \theta L_{k}) \| G_{k} (x_{k})\|^{2} + \langle G_{k}(x_{k}), x_{k} - x^{*}\rangle.
	\end{equation}
	For $k_0$ large enough and $k\geq k_0$, we multiply \eqref{eq: E(k+1)-E(k), Bilevel ProxGradient, line 1a} by $t_{k + 1} - \lambda \theta \geq 0$ and \eqref{eq: E(k+1)-E(k), Bilevel ProxGradient, line 1b} by $\lambda \theta \geq 0$, add the resulting inequalities, sum $ - t_{k + 1} \Psi_{k} (x^{*})$ to both sides and use that $t_{k+1}=t_k + \theta$ to obtain
	\begin{align}
		t_{k + 1} (\Psi_{k} (x_{k + 1}) - \Psi_{k}(x^{*})) \nonumber &\leq t_{k}(\Psi_{k}(x_{k}) - \Psi_{k}(x^{*})) - \theta(\lambda - 1)(\Psi_{k}(x_{k}) - \Psi_{k}(x^{*}))\\
		&\quad - \frac{\theta t_{k + 1}}{2} (2 - \theta L_{k}) \| G_{k}(x_{k})\|^{2} + \lambda \theta \langle G_{k}(x_{k}), x_{k} - x^{*}\rangle. \label{eq: E(k+1)-E(k), Bilevel ProxGradient, line 1c}
	\end{align}
	According to \eqref{eq: x(k+1) expressed in terms of G(xk), Bilevel ProxGradient}, we have 
	\begin{equation}\label{eq:using_constitute_equation_first_order}
		\begin{aligned}
			&\lambda \theta \langle G_{k}(x_{k}), x_{k} - x^{*}\rangle = - \lambda \langle x_{k + 1} - x_{k}, x_{k + 1} - x^{*}\rangle + \lambda \| x_{k + 1} - x_{k}\|^{2},\\
			&\| G_{k}(x_{k})\|^{2} = \frac{1}{\theta^{2}} \| x_{k + 1} - x_{k}\|^{2}.
		\end{aligned}
	\end{equation}
	Now, for line \eqref{eq: E(k+1)-E(k), Bilevel ProxGradient, line 2}, we write
	\begin{align}
		&\: t_{k} (\Psi_{k}(x_{k}) - \Psi_{k}(x^{*})) - t_{k} (\Psi_{k - 1}(x_{k}) - \Psi_{k - 1}(x^{*})) \nonumber\\
		= &\: t_{k} \Big[ \Psi_{k}(x_{k}) - \Psi_{k - 1}(x_{k}) - (\Psi_{k}(x^{*}) - \Psi_{k - 1}(x^{*}))\Big]         =   t_{k}( \varepsilon_{k} - \varepsilon_{k - 1}) (H(x_{k}) - H(x^{*})), \label{eq: E(k+1)-E(k), Bilevel ProxGradient, line 2a}
	\end{align}
	and for line \eqref{eq: E(k+1)-E(k), Bilevel ProxGradient, line 3} we have 
	\begin{equation}\label{eq: E(k+1)-E(k), Bilevel ProxGradient, line 3a}
		\frac{\lambda}{2} \| x_{k + 1} - x^{*}\|^{2} - \frac{\lambda}{2} \| x_{k} - x^{*}\|^{2} = \lambda \langle x_{k + 1} - x_{k}, x_{k + 1} - x^{*}\rangle - \frac{\lambda}{2} \| x_{k + 1} - x_{k}\|^{2}.
	\end{equation}
	Putting \eqref{eq: E(k+1)-E(k), Bilevel ProxGradient, line 1c}, \eqref{eq:using_constitute_equation_first_order}, \eqref{eq: E(k+1)-E(k), Bilevel ProxGradient, line 2a} and \eqref{eq: E(k+1)-E(k), Bilevel ProxGradient, line 3a} together yields for all $k \geq k_0$
	\begin{align*}
		E_{k + 1}^{\lambda} - E_{k}^{\lambda}  \leq &\: - \theta (\lambda - 1) (\Psi_{k}(x_{k}) - \Psi_{k}(x^{*})) - \frac{\theta t_{k+1}}{2} (2 - \theta L_{k}) \frac{\| x_{k + 1} - x_{k}\|^{2}}{\theta^{2}} \\
		& - \lambda\langle x_{k + 1} - x_{k}, x_{k + 1} - x^{*}\rangle + \lambda \| x_{k + 1} - x_{k}\|^{2} \\
		& + t_{k} (\varepsilon_{k} - \varepsilon_{k - 1}) (H(x_{k}) - H(x^{*})) + \lambda \langle x_{k + 1} - x_{k}, x_{k + 1} - x^{*}\rangle - \frac{\lambda}{2} \| x_{k + 1} - x_{k}\|^{2} \\
		= &\: -\theta(\lambda - 1) (F(x_{k}) - F(x^{*})) + \Bigl[ t_{k} (\varepsilon_{k} - \varepsilon_{k - 1}) - \theta (\lambda - 1)\varepsilon_{k}\Bigr] (H(x_{k}) - H(x^{*}))\\
		&- \left[\frac{2 - \theta L_{k}}{2 \theta} t_{k+1} - \frac{\lambda}{2}\right] \| x_{k + 1} - x_{k}\|^{2}.
	\end{align*}
	The fact that $L_{k} \to L_{f}$ as $k\to +\infty$ ensures that the term with $\| x_{k + 1} - x_{k}\|^{2}$ is eventually negative, so we drop it to obtain \eqref{eq:dissipation_first_order}, after possibly increasing $k_0$.
\end{proof}

\begin{rmk}\label{rem:continuous_time_first_order}
The proof of Lemma \ref{lem:dissipation_first_order} is largely inspired by the continuous-time analysis of \eqref{eq:intro_tichonov_flow}, which centers around the Lyapunov energy function
	\begin{equation}\label{eq:lypunov_first_order_continuous}
		E^\lambda(t):=t(\Psi_t(x(t))-\Psi_t(x^*)) + \frac{\lambda}{2}\|x(t)-x^*\|^2, \quad \text{for} \ t \geq t_0,
	\end{equation}
	where $\Psi_{t}(x):=f(x) + \varepsilon(t) h(x)$ is the regularized function \eqref{eq:regularized} with $\varepsilon = \varepsilon(t)$. A similar inequality to \eqref{eq:dissipation_first_order} can easily be obtained for \eqref{eq:lypunov_first_order_continuous}, see Appendix \ref{sec:continuous_time}.
\end{rmk}
\begin{thm}\label{thm: first order algorithm}
    Let $(x_{k})_{k \geq 0}$ be the sequence generated by Algorithm \ref{Bilevel Proximal Gradient algorithm}. Then, the following statements are true:
    \begin{itemize}
        \item Case $\delta > 1$: It holds $F(x_{k}) - \min F = \mathcal{O}(k^{-1})$ as $k \to +\infty$.
        
\item[] \hspace*{-\leftmargin} Suppose Assumption \ref{ass:qualification} holds.

        \item Case $\delta = 1$: If $F$ satisfies Assumption \ref{ass:attouch_czarnecki} with $\eta = 1$, it holds 
        \[
            F(x_{k}) - \min F = o(k^{-1}) \quad \text{and} \quad H(x_{k}) \to \min\nolimits_{\argmin F} H \quad \text{as} \ k\to +\infty.
        \]
        Furthermore, $(x_{k})_{k \geq 0}$ converges weakly towards a solution to the bilevel problem \eqref{eq:bilevel_problem}.
        \item Case $\frac{1}{\rho^\star} < \delta < 1$: If $F$ satisfies Assumption \ref{ass:holderian_growth} with exponent $\rho$, it holds
        \[
            F(x_{k}) - F(x^{*}) = o(k^{-1}) \quad \text{and} \quad |H(x_{k}) - \min\nolimits_{\argmin F} H | = o(k^{-1 + \delta}) \quad \text{as} \ k\to +\infty. 
        \]
        Furthermore, $(x_{k})_{k \geq 0}$ converges weakly to a solution to the bilevel problem \eqref{eq:bilevel_problem}.
    \end{itemize}
\end{thm}
\begin{proof}
Let $x^{*}$ be a solution to the bilevel problem \eqref{eq:bilevel_problem} and $\lambda >1$.  The first part of the proof is  a consequence of Lemma \ref{lem: general descent lemma for discrete case} with $\hat \delta := \delta$, $\eta := 1$, $H_{*} := H(x^{*}) - \inf_{x \in \cH} H(x)$, $C_{2} := \theta (\lambda - 1) >0$, and
$$F_{k} := F(x_{k}) - F(x^*),\quad H_{k} := H(x_{k}) - H(x^*), \quad g_{k} := 0, \quad V_{k} := \frac{\lambda}{2} \| x_{k} - x^{*}\|^{2}$$ for all $k \geq k_0$, where $k_0 \geq 1$ is given by Lemma \ref{lem:dissipation_first_order}.
 As $\varepsilon_k$ is picked as in \eqref{eq:def_epsilon}, $\zeta_{k, \delta} =  \theta(\lambda - 1)\varepsilon_{k} -  t_{k} (\varepsilon_{k} - \varepsilon_{k - 1})$ can be shown to satisfy \eqref{eq:general descent lemma for discrete case, assumption on theta}. Let $k \geq 1$. Indeed, according to the mean value theorem, we have, for some $\xi_k \in (k+\beta -1, k + \beta)$, that  
    \[
        \varepsilon_{k - 1} - \varepsilon_{k} = \frac{c}{(k + \beta - 1)^{\delta}} - \frac{c}{(k + \beta)^{\delta}} = \frac{c \delta}{\xi_{k}^{\delta + 1}},
    \]
    which gives, recalling that $t_k=\theta(k+\gamma)$,
    \[
        \frac{\theta(k + \gamma)}{(k + \beta)^{\delta + 1}} \leq t_{k}(\varepsilon_{k - 1} - \varepsilon_{k}) \leq \frac{\theta(k + \gamma)}{(k + \beta - 1)^{\delta + 1}}.
    \]
    Hence, by possibly increasing $k_0$, the positive constants $C_{1, \zeta}, C_{2, \zeta}$ required by \eqref{eq:general descent lemma for discrete case, assumption on theta} can be easily shown to exist.
    
   The case $\delta > 1$ follows directly from Lemma \ref{lem: general descent lemma for discrete case}. We treat the other two cases simultaneously. We first show that Assumption \ref{ass:attouch_czarnecki} implies condition \eqref{eq: Attouch-Czarnecki condition lemma, discrete}, and that Assumption \ref{ass:holderian_growth} implies conditions \eqref{eq: Hölderian error bound condition lemma, discrete} + \eqref{eq: Attouch-Czarnecki condition lemma, discrete} in Lemma \ref{lem: general descent lemma for discrete case}. Using \eqref{eq: Attouch-Czarnercki implies the integrability condition of the lemma}, we observe that Assumption \ref{ass:attouch_czarnecki} implies that condition \eqref{eq: Attouch-Czarnecki condition lemma, discrete} is satisfied for 
    \[
        I_{k}^{r} := (F - \min F)^{*} ( r \varepsilon_k p^* ) - \sigma_{\argmin F}( r \varepsilon_k p^* ) \quad \text{for all} \ k \geq 0,
    \]
    where $p^{*} \in -\partial H(x^{*})$ and $p^* \in N_{\argmin F}(x^*)$ is given by Assumption \ref{ass:qualification}. For all $r >0$, $rp^* \in N_{\argmin F}(x^*)$, and the sequence $(I_{k}^{r})_{k \geq 0}$ is indeed summable. Using Lemma \ref{lem: Hölderian error bound implies the second condition of the lemma}, Assumption \ref{ass:holderian_growth} implies condition \eqref{eq: Hölderian error bound condition lemma, discrete} for $C_{3} := \| p^{*}\| (\tau\rho^{-1})^{-\frac{1}{\rho}}$, and, using \eqref{eq: bound for conjugates}, also condition \eqref{eq: Attouch-Czarnecki condition lemma, discrete}. Lemma \ref{lem: general descent lemma for discrete case} ensures that, in both cases, the limit $\lim_{k\to +\infty} E_{k}^{\lambda}$ exists. Choosing $1 < \lambda_{1} < \lambda_{2}$, we have 
   \begin{equation}\label{eq:proof_thm_first_order_1}
	E_{k}^{\lambda_{2}} - E_{k}^{\lambda_{1}} = \frac{\lambda_{2} - \lambda_{1}}{2} \| x_{k} - x^{*}\|^{2} \quad \text{for all} \ k \geq 1, \quad \text{thus} \quad \lim_{k \to +\infty} \| x_{k} - x^{*}\|^{2} \quad \text{exists}.
   \end{equation}
Since $x^{*}$ was an arbitrary solution to the bilevel problem \eqref{eq:bilevel_problem}, we have verified the first condition of Opial's Lemma (see Lemma \ref{lem: discrete Opial}). It follows, by definition of $E_k^\lambda$, that
    \[
        \lim_{k\to +\infty} q_{k} := t_{k} (F(x_{k}) - F(x^{*})) + t_k \varepsilon_{k - 1} (H(x_{k}) - H(x^{*})) \quad \text{exists}.
    \]
At the same time, both for the cases $\delta = 1$ and $\frac{1}{\rho^\star} < \delta < 1$, according to the summability statements given by Lemma \ref{lem: general descent lemma for discrete case} we have
	\begin{equation}\label{eq:proof_thm_first_order_2}
	\sum_{k = 1}^{+\infty} \frac{|q_{k}|}{t_{k}} \leq \sum_{k = 1}^{+\infty} (F(x_{k}) - F(x^{*})) + \sum_{k = 1}^{+\infty} \varepsilon_{k - 1} | H(x_{k}) - H(x^{*})| < +\infty. 
	\end{equation}
    Since $\lim_{k\to +\infty} | q_{k}|$ exists as well, it must be the case that $\lim_{k\to +\infty} q_{k} = 0$.

We now establish the weak convergence of $(x_k)_{k \geq 0}$, followed by the derivation of the little-$o$ rates. From \eqref{eq:proof_thm_first_order_1} we get that $(x_k)_{k \geq 0}$ is bounded. Let $x^{\diamond}$ be a weak sequential cluster point of this sequence and $x_{k_j} \rightharpoonup x^{\diamond}$ as $j \to +\infty$. From the convergence rate $F(x_{k}) - F(x^{*}) = \mathcal{O}(k^{-1})$ as $k \to +\infty$ given by Lemma \ref{lem: general descent lemma for discrete case}, and the weak lower semicontinuity of $F$, we deduce that $x^{\diamond} \in \argmin F$. On the other hand,
    \begin{equation}\label{eq:proof_thm_first_order_3}
    	t_{k_j}\varepsilon_{k_{j} - 1} \big(H(x_{k_j}) - H(x^*)\big)\leq t_{k_j} (F(x_{k_j})- F(x^*)) + t_{k_j}\varepsilon_{k_{j} - 1} \big(H(x_{k_j}) - H(x^*)\big) = q_{k_j} \quad \text{for all} \ j \geq 0.
    \end{equation} 
    Thus, passing to the liminf, using that $(t_{k} \varepsilon_{k - 1})_{k \geq 1}$ is bounded away from zero, and the weak lower semicontinuity of $H$, we deduce that $H(x^{\diamond})\leq H(x^*)$. As $x^{\diamond} \in \argmin F$ and $x^*$ is an arbitrary solution to \eqref{eq:bilevel_problem}, $x^{\diamond}$ is also a solution to \eqref{eq:bilevel_problem}. Opial's Lemma (see Lemma \ref{lem: discrete Opial}) yields that the whole sequence $(x_k)_{k \geq 0}$ converges weakly to a solution to the bilevel problem \eqref{eq:bilevel_problem}. Let us denote this solution by $x^*$.
    
To derive the little-$o$ rates, we distinguish between the two cases. If $\delta = 1$, we use the subgradient inequality applied to $H$ with $-p^* \in \partial H(x^*)$. If $\frac{1}{\rho^{\star}} < \delta < 1$, we use the lower bound for $t_{k} \varepsilon_{k - 1}(H(x_{k}) - H(x^{*}))$ provided by Lemma \ref{lem: general descent lemma for discrete case}. For all $k \geq 1$, we have, respectively,
    \begin{align}
    	t_k (F(x_k) -F(x^*)) - t_k \varepsilon_{k - 1} \langle p^*, x_k - x^*\rangle &\leq t_k (F(x_k) -F(x^*)) + t_k \varepsilon_{k - 1} (H(x_k)-H(x^*)) = q_k,\label{eq:proof_thm_first_order_4_1} \\
        t_{k} (F(x_{k}) - F(x^{*})) - C_{H} k^{-\delta + \frac{1}{\rho^{\star}}} &\leq t_{k} (F(x_{k}) - F(x^{*})) + t_{k} \varepsilon_{k - 1} (H(x_{k}) - H(x^{*})) = q_{k}.\label{eq:proof_thm_first_order_4_2}
    \end{align}
    When $\delta = 1$, we have $t_{k} \varepsilon_{k - 1} = \mathcal{O}(1)$ as $k\to +\infty$. Passing to the limit and using $x_k\rightharpoonup x^*$ or $k^{-\delta + \frac{1}{\rho^{\star}}}\to 0$, and $q_k\to0$ as $k\to +\infty$, we get $\lim_{k\to +\infty} t_k (F(x_k)-F(x^*)) = 0$. Consequently, since $q_k\to0$, also the limit of $t_k \varepsilon_{k - 1} (H(x_k)-H(x^*))$ exists and must be zero in both cases.  
\end{proof}

\section{Bilevel Fast Proximal-Gradient method}\label{sec:second_order}

In this section, we study the asymptotic behavior of a semi-implicit discretization of the monotone inclusion form of \eqref{eq:intro_su_boyd_candes_tichonov}. Following the same rationale of \eqref{eq:discretization_of_tichonov}, we consider for some $y_k \in \cH$, 
\begin{equation*}
    \frac{x_{k + 1} - 2x_{k} + x_{k - 1}}{\theta^{2}} + \frac{\alpha}{t_{k + 1}} \frac{x_{k} - x_{k - 1}}{\theta} + \partial \hat f(x_{k+1}) +  \nabla f(y_{k}) + \varepsilon_k \partial h (x_{k+1}) + \varepsilon_{k} \nabla h(y_{k}) \ni 0, \quad \text{for} \ k \geq 1.
\end{equation*}
After multiplication by $s:=\theta^{2}$, we come to 
\begin{equation*}
x_{k + 1} + s\big(\partial \hat f(x_{k+1}) + \varepsilon_k\partial \hat h(x_{k+1})\big) \ni  x_{k} + \left(1 - \frac{\alpha\theta}{t_{k + 1}}\right)(x_{k} - x_{k - 1}) - s \bigl(\nabla f(y_{k}) + \varepsilon_{k} \nabla h(y_{k})\bigr), \quad \text{for} \ k \geq 1,
\end{equation*}
which, setting $t_k:=\theta (k + \gamma)$, denoting $\alpha_k:=1-\frac{\alpha \theta}{t_{k+1}}$, and picking $y_k:=x_{k} + \alpha_{k} (x_{k} - x_{k - 1})$ leads us to Algorithm \ref{alg:second_order}. Employing the same notation as in the first-order setting \eqref{eq:notation_first_order}, we can interpret Algorithm \ref{alg:second_order} as the \emph{Fast Proximal-Gradient algorithm} (also known as FISTA), proposed by Beck and Teboulle in \cite{bt09} and built upon ideas of Nesterov \cite{nesterov83} and G{\"u}ler \cite{guler92}, applied to \eqref{eq:regularized} with $\varepsilon_k\to 0$. As the latter requires that $s\in (0, \frac{1}{L_k}]$ and $\varepsilon_k\to 0$, here we will pick $s \in (0, \frac{1}{L_{\nabla f}})$.

\begin{algorithm}[t]
    \DontPrintSemicolon
    \KwIn{$x_{0}, x_{1} \in\mathcal{H}$, $\alpha > 3$, $\gamma \geq 0$, $0 < s < \frac{1}{L_{f}}$ \hfill \textcolor{gray}{$\rhd$ $\varepsilon_k$ as in \eqref{eq:def_epsilon}}}
        \For{$k = 1, 2, \ldots$}{
            $\displaystyle\alpha_{k} := 1 - \frac{\alpha}{k + \gamma + 1}$\;
            $y_{k} := x_{k} + \alpha_{k}(x_{k} - x_{k - 1})$\;
            $x_{k + 1} := \prox_{s(\hat f + \varepsilon_k \hat h)} \big( y_{k} - s (\nabla f(y_{k}) + \varepsilon_k \nabla h(y_k) ) \big)$ \;
        }
    \caption{Bilevel Fast Proximal-Gradient method}
    \label{alg:second_order}
\end{algorithm}

For analyzing its asymptotic behavior, we make use of the following energy function:
\begin{equation}\label{eq:lyapunov_second_order}
    E_{k}^{\lambda} := t_{k}^{2} (\Psi_{k - 1}(x_{k}) - \Psi_{k - 1}(x^{*})) + \frac{1}{2}\left\| \lambda(x_{k - 1} - x^{*}) + t_{k} \frac{x_{k} - x_{k - 1}}{\theta}\right\|^{2} + \frac{\lambda (\alpha - 1 - \lambda)}{2} \| x_{k - 1} - x^{*}\|^{2}, 
\end{equation}
for $k \geq 1$, where $x^{*}$ is a solution to \eqref{eq:bilevel_problem}, $\theta:=\sqrt{s}$, and $\lambda \in (2, \alpha - 1)$, which is nonempty since $\alpha > 3$. Following closely the Lyapunov analysis of its continuous-time counterpart \eqref{eq:intro_su_boyd_candes_tichonov}, we arrive at the following estimate.
\begin{lem}\label{lem:dissipation_second_order}
	Let $(x_k)_{k \geq 0}$ be the sequence generated by Algorithm \ref{alg:second_order} and $(E_k^\lambda)_{k \geq 1}$ be defined as in \eqref{eq:lyapunov_second_order}. Then, there exists $k_0\geq 1$ such that for all $k\geq k_0$
	\begin{equation}\label{eq:dissipation_second_order}
	\begin{aligned}
		E_{k + 1}^{\lambda} - E_{k}^{\lambda} + (\alpha - 1 - \lambda) \theta t_{k + 1} \alpha_{k} \left\| \frac{x_{k} - x_{k - 1}}{\theta}\right\|^{2} & + \zeta_{k, \delta} (H(x_{k}) - H(x^{*}))\\
		\leq &  -\theta^{2} (\lambda - 2) k  (F(x_{k}) - F(x^{*})),
	\end{aligned}
	\end{equation}
where $(\zeta_{k, \delta})_{k\geq 1}$ is a sequence defined by
	\[
	\zeta_{k, \delta} := \theta \big( (\lambda - 2)t_k + \theta(\lambda - 1)\big) \varepsilon_k - t_k^2 (\varepsilon_k - \varepsilon_{k-1}).
	\]
\end{lem}
\begin{proof}
	 Let $k \geq 1$. First, let us denote for brevity $v_{k} := \lambda(x_{k - 1} - x^{*}) + t_{k} \frac{x_{k} - x_{k - 1}}{\theta}$. We have 
	\begin{align}
		E_{k + 1}^{\lambda} - E_{k}^{\lambda} = &\: t_{k + 1}^{2} (\Psi_{k}(x_{k + 1}) - \Psi_{k}(x^{*})) - t_{k}^{2} (\Psi_{k}(x_{k}) - \Psi_{k}(x^{*})) \label{eq: E_(k)' general line 1}\\
		&+ t_{k}^{2} (\Psi_{k}(x_{k}) - \Psi_{k}(x^{*})) - t_{k}^{2} (\Psi_{k - 1}(x_{k}) - \Psi_{k - 1}(x^{*})) \label{eq: E_(k)' general line 2}\\
		&+ \frac{1}{2} \| v_{k + 1}\|^{2} - \frac{1}{2} \| v_{k}\|^{2} \label{eq: E_(k)' general line 3}\\
		&+ \frac{\lambda(\alpha - 1 - \lambda)}{2} \| x_{k} - x^{*}\|^{2} - \frac{\lambda(\alpha - 1 - \lambda)}{2} \| x_{k - 1} - x^{*}\|^{2}. \label{eq: E_(k)' general line 4}
	\end{align}
We will work with each line separately. For line \eqref{eq: E_(k)' general line 1}, we use Lemma \ref{lem: proximal step}. The proximal step in Algorithm \ref{alg:second_order} can be written as
	\begin{equation}\label{eq: x(k+1) expressed in terms of G(yk), BiFISTA Algorithm}
		x_{k + 1} = y_{k} - sG_{k}(y_{k}), \quad \text{where} \quad G_{k}(y) := \frac{1}{s} \Bigl[ y - \prox_{s \hat{\Phi}_{k}} (y - s \nabla \Phi_{k}(y))\Bigr].
	\end{equation}
Since $\Phi_{k}$ is $L_{k}$-smooth, according to Lemma \ref{lem: proximal step}, for all $x, y\in\mathcal{H}$ it holds
	\begin{equation*}
		\Psi_{k}(y - sG_{k}(y)) \leq  \Psi_{k}(x) - \frac{s}{2} (2 - sL_{k}) \| G_{k}(y)\|^{2} + \langle G_{k}(y), y - x\rangle.
	\end{equation*}
We apply the inequality twice. First, set $y := y_{k}$ and $x := x_{k}$. We obtain
	\begin{equation}\label{eq: E_(k)', x=xk, y=xk}
		\Psi_{k}(x_{k + 1}) \leq  \Psi_{k}(x_{k}) - \frac{s}{2}(2 - sL_{k}) \| G_{k}(y_{k})\|^{2} + \langle G_{k}(y_{k}), y_{k} - x_{k}\rangle.
	\end{equation}
	Now, we set $y := y_{k}$ and $x := x^{*}$. This yields 
	\begin{equation} \label{eq: E_(k)', x=xk, y=x*}
		\Psi_{k}(x_{k + 1}) \leq \Psi_{k}(x^{*}) - \frac{s}{2}(2 - sL_{k}) \| G_{k}(y_{k})\|^{2} + \langle G_{k}(y_{k}), y_{k} - x^{*}\rangle.
	\end{equation}
From now on, we will use that $s = \theta^{2}$. For $k_0$ large enough and $k \geq k_0$, we multiply \eqref{eq: E_(k)', x=xk, y=xk} by $t_{k + 1} - \lambda\theta \geq 0$, and  \eqref{eq: E_(k)', x=xk, y=x*} by $\lambda\theta \geq 0$, add the resulting inequalities and sum $-t_{k + 1}\Psi_{k}(x^{*})$ to both sides to obtain
	\begin{align*}
		t_{k + 1} (&\Psi_{k}(x_{k + 1}) - \Psi_{k}(x^{*})) \leq (t_{k + 1} - \lambda\theta) (\Psi_{k}(x_{k}) - \Psi_{k}(x^{*}))\\
		&\quad  + \big\langle G_{k}(y_{k}), (t_{k + 1} - \lambda\theta)(y_{k} - x_{k}) + \lambda\theta (y_{k} - x^{*})\big\rangle - \frac{\theta^{2} t_{k + 1}}{2}(2 - \theta^{2}L_{k}) \|G_{k}(y_{k})\|^{2}.
	\end{align*}
We multiply both sides of the previous inequality by $t_{k + 1}$, and add and subtract $t_{k}^{2} (\Psi_{k}(x_{k}) - \Psi_{k}(x^{*}))$, which yields 
	\begin{align}
		t_{k + 1}^{2}&(\Psi_{k}(x_{k + 1}) - \Psi_{k}(x^{*})) \leq \Bigl[ t_{k + 1}^{2} - \lambda\theta t_{k + 1} - t_{k}^{2}\Bigr] (\Psi_{k}(x_{k}) - \Psi_{k}(x^{*})) + t_{k}^{2} (\Psi_{k}(x_{k}) - \Psi_{k}(x^{*})) \nonumber\\
		&+ t_{k + 1} \Bigl\langle G_{k}(y_{k}), (t_{k + 1} - \lambda\theta)(y_{k} - x_{k}) + \lambda\theta (y_{k} - x^{*})\Bigr\rangle - \frac{\theta^{2} t_{k + 1}^{2}}{2}(2 - \theta^{2} L_{k}) \| G_{k}(y_{k})\|^{2}. \label{eq: E_(k)' general line 1 bis}
	\end{align}
	Continuing with line \eqref{eq: E_(k)' general line 2}, we have 
	\begin{align}
		& \ t_{k}^{2}(\Psi_{k}(x_{k}) - \Psi_{k}(x^{*})) - t_{k}^{2} (\Psi_{k - 1}(x_{k}) - \Psi_{k - 1}(x^{*}))\\
		= & \ t_{k}^{2} \Bigl[\Psi_{k}(x_{k}) - \Psi_{k - 1}(x_{k}) - (\Psi_{k}(x^{*}) - \Psi_{k - 1}(x^{*}))\Bigr] \nonumber \\
		= & \ \: t_{k}^{2} (\varepsilon_{k} - \varepsilon_{k - 1}) (H(x_{k}) - H(x^{*})). \label{eq: E_(k)' general line 2 bis}
	\end{align}
	For line \eqref{eq: E_(k)' general line 3}, from \eqref{eq: x(k+1) expressed in terms of G(yk), BiFISTA Algorithm} first we notice that 
	\begin{align*}
		v_{k + 1} - v_{k} &= \lambda (x_{k} - x^{*}) + \frac{t_{k + 1}}{\theta}(x_{k + 1} - x_{k}) - \lambda (x_{k - 1} - x^{*}) - \frac{t_{k}}{\theta} (x_{k} - x_{k - 1}) \\
		&= \frac{t_{k + 1}}{\theta} (x_{k + 1} - x_{k}) - \left(\frac{t_{k}}{\theta} - \lambda\right) (x_{k} - x_{k - 1}) \\
		&= \frac{t_{k + 1}}{\theta} (x_{k + 1} - x_{k}) - \left[ \frac{t_{k + 1}\alpha_{k}}{\theta} + (\alpha - 1 - \lambda)\right](x_{k} - x_{k - 1}) \\
		&= \frac{t_{k + 1}}{\theta}\Bigl[ x_{k + 1} - \Bigl( x_{k} + \alpha_{k} (x_{k} - x_{k - 1})\Bigr)\Bigr] - (\alpha - 1 - \lambda) (x_{k} - x_{k - 1}) \\
		&= \frac{t_{k + 1}}{\theta} (x_{k + 1} - y_{k}) - (\alpha - 1 - \lambda)(x_{k} - x_{k - 1}) \\
		&= - \theta t_{k + 1} G_{k}(y_{k}) - (\alpha - 1 - \lambda) (x_{k} - x_{k - 1}).
	\end{align*}
	Therefore, 
	\begin{align}
		\langle v_{k + 1} - v_{k}, v_{k + 1}\rangle = & \left\langle -\theta t_{k + 1} G_{k}(y_{k}) - (\alpha - 1 - \lambda)(x_{k} - x_{k - 1}), \lambda(x_{k} - x^{*}) + t_{k + 1} \frac{x_{k + 1} - x_{k}}{\theta}\right\rangle \nonumber\\
		= & -\theta t_{k + 1} \left\langle G_{k}(y_{k}), \lambda(x_{k} - x^{*}) + t_{k + 1} \frac{x_{k + 1} - x_{k}}{\theta}\right\rangle \nonumber\\
		& -(\alpha - 1 - \lambda) \left\langle x_{k} - x_{k - 1}, \lambda(x_{k} - x^{*}) + t_{k + 1} \frac{x_{k + 1} - x_{k}}{\theta}\right\rangle. \label{eq: <v(k+1)-vk,v(k+1)> line 2}
	\end{align}
	From Algorithm \ref{alg:second_order} and \eqref{eq: x(k+1) expressed in terms of G(yk), BiFISTA Algorithm}, we know that 
	\[
	\frac{x_{k + 1} - x_{k}}{\theta} = -\theta G_{k}(y_{k}) + \alpha_{k} \frac{x_{k} - x_{k - 1}}{\theta}.
	\]
Using this equality, we can write line \eqref{eq: <v(k+1)-vk,v(k+1)> line 2} as 
	\begin{align*}
		& - ( \alpha - 1 - \lambda) \left\langle x_{k} - x_{k - 1}, \lambda (x_{k} - x^{*}) + t_{k + 1} \frac{x_{k + 1} - x_{k}}{\theta}\right\rangle \\
		= & - \lambda (\alpha - 1 - \lambda) \langle x_{k} - x_{k - 1}, x_{k} - x^{*}\rangle + (\alpha - 1 - \lambda) \theta t_{k + 1} \langle x_{k} - x_{k - 1}, G_{k}(y_{k})\rangle \\
		& - (\alpha - 1 - \lambda) \theta t_{k + 1} \alpha_{k} \left\| \frac{x_{k} - x_{k - 1}}{\theta}\right\|^{2}.
	\end{align*}
	We have 
	\begin{align*}
		& \|v_{k + 1} - v_{k}\|^{2} =\: \Bigl\| - \theta t_{k + 1} G_{k}(y_{k}) - (\alpha - 1 - \lambda) (x_{k} - x_{k - 1})\Bigr\|^{2} \\
		= & \ \theta^{2} t_{k + 1}^{2} \| G_{k}(y_{k})\|^{2} + 2\theta t_{k + 1} (\alpha - 1 - \lambda) \langle  G_{k}(y_{k}),  x_{k} - x_{k - 1}\rangle  + (\alpha - 1 - \lambda)^{2} \| x_{k} - x_{k - 1}\|^{2}. 
	\end{align*}
	Putting everything together gives 
	\begin{align}
		&\frac{1}{2} \| v_{k + 1}\|^{2} - \frac{1}{2} \| v_{k}\|^{2} = \langle v_{k + 1} - v_{k}, v_{k + 1}\rangle - \frac{1}{2} \| v_{k + 1} - v_{k}\|^{2} \nonumber\\
		= & -\theta t_{k + 1} \left\langle G_{k}(y_{k}), \lambda(x_{k} - x^{*}) + t_{k + 1} \frac{x_{k + 1} - x_{k}}{\theta}\right\rangle \nonumber\\
		&- \lambda(\alpha - 1 - \lambda) \langle x_{k} - x_{k - 1}, x_{k} - x^{*}\rangle + (\alpha - 1 - \lambda) \theta t_{k + 1} \langle G_{k}(y_{k}), x_{k} - x_{k - 1}\rangle \nonumber \\
		&- (\alpha - 1 - \lambda) \theta t_{k + 1} \alpha_{k} \left\| \frac{x_{k} - x_{k - 1}}{\theta}\right\|^{2} - \frac{\theta^{2} t_{k + 1}^{2}}{2} \| G_{k}(y_{k})\|^{2} \nonumber\\
		&- \theta t_{k + 1} (\alpha - 1 - \lambda) \langle G_{k}(y_{k}), x_{k} - x_{k - 1}\rangle - \frac{(\alpha - 1 - \lambda)^{2}}{2} \| x_{k} - x_{k - 1}\|^{2}. \label{eq: E_(k)' general line 3 bis}
	\end{align}
	For line \eqref{eq: E_(k)' general line 4}, it holds 
	\begin{align}
		& \ \frac{\lambda(\alpha - 1 - \lambda)}{2} \| x_{k} - x^{*}\|^{2} - \frac{\lambda(\alpha - 1 - \lambda)}{2} \| x_{k - 1} - x^{*}\|^{2} \nonumber \\
		= & \ \lambda(\alpha - 1 - \lambda) \langle x_{k} - x_{k - 1}, x_{k} - x^{*}\rangle - \frac{\lambda (\alpha - 1 - \lambda)}{2} \| x_{k} - x_{k - 1}\|^{2}. \label{eq: E_(k)' general line 4 bis}
	\end{align}
	We now combine \eqref{eq: E_(k)' general line 1 bis}, \eqref{eq: E_(k)' general line 2 bis}, \eqref{eq: E_(k)' general line 3 bis} and \eqref{eq: E_(k)' general line 4 bis}. We cancel out the corresponding terms, drop some nonpositive terms, again take into account that $y_{k} - x_{k + 1} = \theta^{2} G_{k}(y_{k})$ as per \eqref{eq: x(k+1) expressed in terms of G(yk), BiFISTA Algorithm}, and observe that, since $\lambda \geq 2$, simple algebraic computations yield
	\begin{equation}\label{eq: formula for t(k+1)^2-(lambda.theta)t(k+1)-tk^2}
		t_{k + 1}^{2} - \lambda\theta t_{k + 1} - t_{k}^{2} =\theta \Bigl[(2 - \lambda)t_k + \theta(1 - \lambda)\Bigl]  \leq 0 \quad \text{for all} \ k \geq 1.   
	\end{equation}
This leads us for all $k \geq k_0$ to
	\begin{align*}
		&\: E_{k + 1}^{\lambda} - E_{k}^{\lambda} \\
		\leq &\: \Bigl[ t_{k + 1}^{2} - \lambda\theta t_{k + 1} - t_{k}^{2}\Bigr] (\Psi_{k}(x_{k}) - \Psi_{k}(x^{*})) + t_{k + 1} \Bigl\langle G_{k}(y_{k}), (t_{k + 1} - \lambda\theta)(y_{k} - x_{k}) + \lambda\theta (y_{k} - x^{*})\Bigr\rangle  \\
		&- \frac{\theta^{2} t_{k + 1}^{2}}{2} (2 - \theta^{2} L_{k}) \| G_{k}(y_{k})\|^{2} + t_{k}^{2} (\varepsilon_{k} - \varepsilon_{k - 1}) (H(x_{k}) - H(x^{*}))\\
		&-\theta t_{k + 1} \left\langle G_{k}(y_{k}), \lambda(x_{k} - x^{*}) + t_{k + 1} \frac{x_{k + 1} - x_{k}}{\theta}\right\rangle - \lambda(\alpha - 1 - \lambda) \langle x_{k} - x_{k - 1}, x_{k} - x^{*}\rangle \\
		&+ (\alpha - 1 - \lambda) \theta t_{k + 1} \langle G_{k}(y_{k}), x_{k} - x_{k - 1}\rangle - (\alpha - 1 - \lambda) \theta t_{k + 1} \alpha_{k} \left\| \frac{x_{k} - x_{k - 1}}{\theta}\right\|^{2}  \\
		&- \frac{\theta^{2} t_{k + 1}^{2}}{2} \| G_{k}(y_{k})\|^{2} - \theta t_{k + 1} (\alpha - 1 - \lambda) \langle G_{k}(y_{k}), x_{k} - x_{k - 1}\rangle  - \frac{(\alpha - 1 - \lambda)^{2}}{2} \| x_{k} - x_{k - 1}\|^{2} \\
		&+ \lambda(\alpha - 1 - \lambda) \langle x_{k} - x_{k - 1}, x_{k} - x^{*}\rangle - \frac{\lambda(\alpha - 1 - \lambda)}{2} \| x_{k} - x_{k - 1}\|^{2}  \\
		\leq &\: \theta \Bigl[(2 - \lambda)t_k + \theta(1 - \lambda)\Bigl]  (\Psi_{k}(x_{k}) - \Psi_{k}(x^{*})) + t_{k + 1} \Bigl\langle G_{k}(y_{k}), t_{k + 1}(y_{k} - x_{k + 1})\Bigr\rangle \\
		&- \frac{\theta^{2} t_{k + 1}^{2}}{2} (3 - \theta^{2} L_{k}) \| G_{k}(y_{k})\|^{2} + t_{k}^{2} (\varepsilon_{k} - \varepsilon_{k - 1}) (H(x_{k}) - H(x^{*})) \\
		&- (\alpha - 1 - \lambda) \theta t_{k + 1} \alpha_{k} \left\| \frac{x_{k} - x_{k - 1}}{\theta}\right\|^{2} \\
		\leq &\: - \theta  (\lambda - 2) t_k (F(x_{k}) - F(x^{*})) - \frac{\theta^{2} t_{k + 1}^{2}}{2}(1 - \theta^{2}L_{k}) \| G_{k}(y_{k})\|^{2} - (\alpha - 1 - \lambda) \theta t_{k + 1} \alpha_{k} \left\| \frac{x_{k} - x_{k - 1}}{\theta}\right\|^{2} \\
		&+ \theta \Bigl[ \big( (2 - \lambda)t_k + \theta(1 - \lambda)\big) \varepsilon_k + t_k^2 \frac{\varepsilon_k - \varepsilon_{k-1}}{\theta}\Bigl] (H(x_{k}) - H(x^{*})). 
	\end{align*}
By dropping the term with $-(1 - \theta^{2} L_{k})$, which is eventually nonpositive, since $\theta^{2} < \frac{1}{L_{f}}$ and $L_{k} \to L_{f}$ as $k\to +\infty$, and by taking into account that $t_{k} = \theta(k + \gamma) \geq \theta k$, this yields \eqref{eq:dissipation_second_order}, after possibly increasing $k_0$.
\end{proof}

\begin{rmk}\label{rem:continuous_time_second_order}
	Even in this setting, the proof of Lemma \ref{lem:dissipation_second_order} is largely inspired from that of the continuous-time system \eqref{eq:intro_su_boyd_candes_tichonov}, which can be studied with the Lyapunov energy function
	\begin{equation}\label{eq:lyapunov_second_order_continuous}
		E^\lambda(t):=t^2(\Psi_t(x(t)) - \Psi_t(x^*)) + \frac{1}{2}\|\lambda (x(t)-x^*) + t \dot x(t)\|^2 + \frac{\lambda(\alpha-1-\lambda)}{2}\|x(t)-x^*\|^2, \ \text{for} \ t \geq t_0,
	\end{equation}
where we recall that $\Psi_t(x)=f(x) + \varepsilon(t)h(x)$ is the regularized function introduced in \eqref{eq:regularized}. Observe the striking similarity with its discrete-time counterpart \eqref{eq:lyapunov_second_order}. A similar result to Lemma \ref{lem:dissipation_second_order} can be established for \eqref{eq:lyapunov_second_order_continuous}, see Appendix \ref{sec:continuous_time}.
\end{rmk}

\begin{thm}\label{thm:second_order_main}
    Let $(x_{k})_{k \geq 0}$ be the sequence generated by Algorithm \ref{alg:second_order}. Then, the following statements are true:
    \begin{itemize}
        \item Case $\delta > 2$: It holds $F(x_k) - \min F = \mathcal{O} (k^{-2})$ as $k\to +\infty$.
\item[] \hspace*{-\leftmargin} Suppose Assumption \ref{ass:qualification} holds.
        \item Case $\delta = 2$: If $F$ satisfies Assumption \ref{ass:attouch_czarnecki} with $\eta = 2$, it holds
        \[
            F(x_{k}) - \min F = o(k^{-2}) \quad \text{and} \quad H(x_{k}) \to \min\nolimits_{\argmin F} H \quad \text{as} \ k\to +\infty.
        \]
        Furthermore, $(x_{k})_{k\geq 0}$ converges weakly towards a solution to the bilevel problem \eqref{eq:bilevel_problem}.
        
        \item Case $\frac{2}{\rho^\star} < \delta < 2$: If $F$ satisfies Assumption \ref{ass:holderian_growth} with exponent $\rho$, it holds
        \[
            F(x_{k}) - \min F = o(k^{-2}) \quad \text{and} \quad |H(x_{k}) -\min\nolimits_{\argmin F} H| = o(k^{-2 + \delta}) \quad \text{as} \ k\to +\infty.
        \]
        Furthermore, $(x_{k})_{k \geq 0}$ converges weakly towards a solution to the bilevel problem \eqref{eq:bilevel_problem}.
    \end{itemize}
\end{thm}
\begin{proof}
Let $x^{*}$ be a solution to the bilevel problem \eqref{eq:bilevel_problem} and $\lambda \in (2, \alpha - 1])$.
The first part of the proof is a consequence of Lemma \ref{lem: general descent lemma for discrete case} with $\hat \delta := \frac{\delta}{2}$, $\eta=2$, $\theta := \sqrt{s}$, $H_{*} := H(x^{*}) - \inf_{x \in \cH} H(x)$, $C_{2} := \theta^{2} (\lambda - 2) >0$, and 
$$F_{k} := F(x_{k}) - F(x^{*}), \quad H_{k} := H(x_{k}) - H(x^{*}), \quad g_{k} := (\alpha - 1 - \lambda)\theta t_{k + 1} \alpha_{k} \left\| \frac{x_{k} - x_{k - 1}}{\theta}\right\|^{2},$$ $$V_{k} := \frac{1}{2} \| v_{k}\|^{2} + \frac{\lambda(\alpha - 1 - \lambda)}{2}\| x_{k - 1} - x^{*}\|^{2}$$
for $k \geq k_0$, where $k_0$ is given by Lemma \ref{lem:dissipation_second_order}. Arguing similarly as in the proof for Theorem \ref{thm: first order algorithm}, we can produce positive constants $C_{1, \zeta}, C_{2, \zeta}$ such that \eqref{eq:general descent lemma for discrete case, assumption on theta} holds.
    
The case $\delta > 2$ follows directly from Lemma \ref{lem: general descent lemma for discrete case}. As in the proof of Theorem \ref{thm: first order algorithm}, we treat the other two cases simultaneously. Similar arguments show that Assumption \ref{ass:attouch_czarnecki} implies condition \eqref{eq: Attouch-Czarnecki condition lemma, discrete}, and that Assumption \ref{ass:holderian_growth} implies conditions \eqref{eq: Hölderian error bound condition lemma, discrete} $+$ \eqref{eq: Attouch-Czarnecki condition lemma, discrete}, even with $\eta= 2$. Lemma \ref{lem: general descent lemma for discrete case} ensures that, in both cases, the limit $\lim_{k \to +\infty} E_{k}^{\lambda}$ exists. Choosing $2 < \lambda_{1} < \lambda_{2} < \alpha - 1$, we have
    \begin{align*}
        E_{k}^{\lambda_{2}} - E_{k}^{\lambda_{1}} = & \ \frac{1}{2} \left\| \lambda_{2}(x_{k - 1} - x^{*}) + t_{k}\frac{x_{k} - x_{k - 1}}{\theta}\right\|^{2} + \frac{\lambda_{2}(\alpha - 1 - \lambda_{2})}{2} \| x_{k - 1} - x^{*}\|^{2} \\
        &- \frac{1}{2} \left\| \lambda_{1}(x_{k - 1} - x^{*}) + t_{k}\frac{x_{k} - x_{k - 1}}{\theta}\right\|^{2} - \frac{\lambda_{1}(\alpha - 1 - \lambda_{1})}{2} \| x_{k - 1} - x^{*}\|^{2} \\
        = & \ \frac{\lambda_{2}^{2}}{2} \| x_{k - 1} - x^{*}\|^{2} + \lambda_{2}t_{k} \left\langle x_{k - 1} - x^{*}, \frac{x_{k} - x_{k - 1}}{\theta}\right\rangle + \frac{t_{k}^{2}}{2} \left\| \frac{x_{k} - x_{k - 1}}{\theta}\right\|^{2} \\
        &+ \left[ \frac{\lambda_{2}(\alpha - 1)}{2} - \frac{\lambda_{2}^{2}}{2}\right] \| x_{k - 1} - x^{*}\|^{2} - \frac{\lambda_{1}^{2}}{2}\| x_{k - 1} - x^{*}\|^{2} - \lambda_{1}t_{k} \left\langle x_{k - 1} - x^{*}, \frac{x_{k} - x_{k - 1}}{\theta}\right\rangle \\
        &- \frac{t_{k}^{2}}{2} \left\| \frac{x_{k} - x_{k - 1}}{\theta}\right\|^{2} - \left[ \frac{\lambda_{1}(\alpha - 1)}{2} - \frac{\lambda_{1}^{2}}{2}\right] \| x_{k - 1} - x^{*}\|^{2} \\
        = & \ (\lambda_{2} - \lambda_{1}) \underbrace{\left[ \frac{\alpha - 1}{2} \| x_{k - 1} - x^{*}\|^{2} + t_{k} \left\langle x_{k - 1} - x^{*}, \frac{x_{k} - x_{k - 1}}{\theta}\right\rangle\right]}_{=: p_{k}} \quad \text{for all} \ k \geq 1.
    \end{align*}
    Since $\lambda_{2} - \lambda_{1} \neq 0$, and $E_k^{\lambda_1}, E_k^{\lambda_2}$ admit a limit, also $\lim_{k\to +\infty} p_{k}$ exists. We now rewrite $E_{k}^{\lambda_{1}}$ as 
    \begin{align*}
        E_{k}^{\lambda_{1}} = &\: t_{k}^{2} (\Psi_{k - 1}(x_{k}) - \Psi_{k - 1}(x^{*})) + \frac{\lambda_{1}^{2}}{2} \| x_{k - 1} - x^{*}\|^{2} + \lambda_{1} t_{k} \left\langle x_{k - 1} - x^{*}, \frac{x_{k - 1} - x_{k}}{\theta}\right\rangle \\
        &+ \frac{t_{k}^{2}}{2} \left\| \frac{x_{k - 1} - x_{k}}{\theta}\right\|^{2} + \left[ \frac{\lambda_{1} (\alpha - 1)}{2} - \frac{\lambda_{1}^{2}}{2}\right] \| x_{k - 1} - x^{*}\|^{2} \\
        = &\: t_{k}^{2} (\Psi_{k - 1}(x_{k}) - \Psi_{k - 1}(x^{*})) + \frac{t_{k}^{2}}{2} \left\|\frac{x_{k} - x_{k - 1}}{\theta}\right\|^{2} + \lambda_{1} p_{k} \\
        = &\: \underbrace{t_{k}^{2} ( F(x_{k}) - F(x^{*})) + t_{k}^{2} \varepsilon_{k-1}(H(x_{k}) - H(x^{*})) + \frac{t_{k}^{2}}{2} \left\| \frac{x_{k} - x_{k - 1}}{\theta}\right\|^{2}}_{=: q_{k}} + \lambda_{1} p_{k} \quad \text{for all} \ k \geq 1.
    \end{align*}
    Since both $\lim_{k\to +\infty} E_{k}^{\lambda_{1}}$ and $\lim_{k\to +\infty} p_{k}$ exist, so does $\lim_{k\to +\infty} q_{k}$. Both for the cases $\delta = 2$ and $\frac{2}{\rho^\star} < \delta < 2$, according to the summability statements given by Lemma \ref{lem: general descent lemma for discrete case} together with the fact that $t_k \varepsilon_{k-1} =\mathcal{O}(k^{1 - \delta})$ as $k \to +\infty$, we have
    \begin{align*}
        \sum_{k = 1}^{+\infty} \frac{|q_{k}|}{t_{k}} &\leq \sum_{k = 1}^{+\infty} t_{k}(F(x_{k}) - F(x^{*})) + \sum_{k = 1}^{+\infty} t_{k}\varepsilon_{k-1} |H(x_{k}) - H(x^{*})| + \frac{1}{2} \sum_{k = 1}^{+\infty} t_{k} \left\| \frac{x_{k} - x_{k - 1}}{\theta}\right\|^{2} <+\infty.
    \end{align*}
Since $\lim_{k\to +\infty} | q_{k}|$ exists as well, it must be the case that $\lim_{k\to +\infty} q_{k} = 0$. Observe that, as  $\left(t_{k} \big\| \frac{x_{k} - x_{k - 1}}{\theta}\big\|^{2}\right)_{k\geq 1}$ is summable, its limit exists and is zero. If we set $\varphi_{k} := \frac{1}{2} \| x_{k - 1} - x^{*}\|^{2}$, we have for all $k \geq 1$
    \begin{align*}
        & (\alpha - 1) \varphi_{k + 1} + \frac{t_{k} - \theta(\alpha - 1)}{\theta} (\varphi_{k + 1} - \varphi_{k}) = (\alpha - 1) \varphi_{k} + \frac{t_{k}}{\theta} (\varphi_{k + 1} - \varphi_{k}) \\
        = & \frac{\alpha - 1}{2} \| x_{k - 1} - x^{*}\|^{2} + \frac{t_{k}}{\theta} \left[ \langle x_{k} - x_{k - 1}, x_{k} - x^{*}\rangle - \frac{1}{2} \| x_{k} - x_{k - 1}\|^{2}\right] \\
        = & \frac{\alpha - 1}{2} \| x_{k - 1} - x^{*}\|^{2} + \frac{t_{k}}{\theta} \left[ \langle x_{k} - x_{k - 1}, x_{k - 1} - x^{*}\rangle + \frac{1}{2} \| x_{k} - x_{k - 1}\|^{2}\right] = p_{k} + \frac{t_{k}}{2\theta} \| x_{k} - x_{k - 1}\|^{2}.
    \end{align*}
Therefore, according to Lemma \ref{lem: for the existence of lim|xk-x*|}, it must be the case that $\lim_{k\to +\infty} \varphi_{k}$ exists as well. As $x^{*}$ is an arbitrary solution to \eqref{eq:bilevel_problem}, we have verified the first condition of Opial's Lemma (see Lemma \ref{lem: discrete Opial}).
     
We proceed to show weak convergence of $(x_k)_{k \geq 0}$ and then little-$o$ rates. As $\lim_{k\to+\infty}\varphi_k \in \R$, we get that $(x_k)_{k \geq 0}$ is bounded. Let $x^{\diamond}$ be a weak sequential cluster point and $x_{k_j} \rightharpoonup x^{\diamond}$ as $j \to +\infty$. From the convergence rate $F(x_{k}) - F(x^{*}) = \mathcal{O}(k^{-2})$ as $k \to +\infty$ given by Lemma \ref{lem: general descent lemma for discrete case}, and the weak lower semicontinuity of $F$, we deduce that $x^{\diamond} \in \argmin F$. On the other hand, denoting by $\hat g_k:= \tfrac{t_{k}^2}{2\theta^2}\|x_k - x_{k-1}\|^2$, we have
     \begin{equation}
     	t_{k_j}^2\varepsilon_{k_{j}-1} \big(H(x_{k_j}) + H(x^*)\big)\leq t_{k_j}^2 (F(x_{k_j})- F(x^*)) + t_{k_j}^2\varepsilon_{k_{j}-1} \big(H(x_{k_j}) + H(x^*)\big) + \hat g_{k_j} = q_{k_j} \quad \text{for all} \ j \geq 0.
     \end{equation} 
Thus, passing to the liminf, using that $(t_k^2 \varepsilon_{k-1})_{k \geq 1}$ is bounded away from zero, and the weak lower semicontinuity of $H$, we deduce that $H(x^{\diamond})\leq H(x^*)$. As $x^{\diamond} \in \argmin F$ and $x^*$ is an arbitrary solution to \eqref{eq:bilevel_problem}, $x^{\diamond}$ is also a solution to \eqref{eq:bilevel_problem}. Opial's Lemma (see Lemma \ref{lem: discrete Opial}) yields that the whole sequence $(x_k)_{k \geq 0}$ converges weakly to a solution to the bilevel problem \eqref{eq:bilevel_problem}. Let us denote this solution by $x^*$ and consider the above estimates with respect to this specific solution.

To derive little-$o$ rates, we proceed in a similar way as in the proof for Algorithm \ref{alg:first_order}. If $\delta = 2$, we use the subgradient inequality applied to $H$ with $-p^* \in \partial H(x^*)$. If $\frac{2}{\rho^{\star}} < \delta < 2$, we use the lower bound for $t_{k}^{2} \varepsilon_{k - 1} (H(x_{k}) - H(x^{*}))$ provided by Lemma \ref{lem: general descent lemma for discrete case}. For all $k \geq 1$, we have, respectively, 
     \begin{align}
     	t_{k}^{2} (F(x_{k}) - F(x^{*})) + \hat{g}_{k} &\leq q_{k} + t_{k}^{2} \varepsilon_{k - 1} \langle p^{*}, x_{k} - x^{*}\rangle, \\
        t_{k}^{2} (F(x_{k}) - F(x^{*})) + \hat{g}_{k} &\leq q_{k} + C_{H} k^{-\delta + \frac{2}{\rho^{\star}}}.
     \end{align}
When $\delta = 2$, we have $t_{k}^{2} \varepsilon_{k - 1} = \mathcal{O}(1)$ as $k\to +\infty$. Passing to the limit and using $x_k\rightharpoonup x^*$ or $k^{-\delta + \frac{2}{\rho^{\star}}} \to 0$, and $q_k \to 0$ as $k\to +\infty$, we get $\lim_{k\to+\infty} t_k^2 (F(x_k)-F(x^*)) = 0$. Consequently, since $q_k \to 0$, also the limit of $t_k^2 \varepsilon_{k-1} (H(x_k)-H(x^*))$ exists and must be zero in both cases. 
\end{proof}

\section{Discussion}\label{sec:discussion}

In Sections \ref{sec:first_order} and \ref{sec:second_order}, we analyzed the asymptotic behavior of Algorithms \ref{alg:first_order} and \ref{alg:second_order} under the geometric assumptions illustrated in Figure \ref{fig:geometric_setting}. In what follows, we present some additional results that arise from our analysis. Although these were not included in the main theorems, they are nonetheless noteworthy. For the sake of brevity, we focus each time on either Algorithm \ref{alg:second_order} or Algorithm \ref{alg:first_order}, as similar arguments apply to both cases.

\paragraph{Ergodic convergence rates.} Let us first drop all geometric assumptions and consider Algorithm \ref{alg:second_order}. When $0 < \delta < 2$, starting from the dissipativity estimate in Lemma \ref{lem:dissipation_second_order} and arguing as in \eqref{eq: pre rates for Fk, case 0<delta<=1, lemma, discrete} we obtain
\begin{equation}\label{eq:discussion_rates_1}
    F(x_{k}) - \min F = \mathcal{O}(k^{-\delta}) \quad \text{and} \quad \min_{k_{0}\leq \ell \leq k + 1} \{H(x_{\ell}) - \min\nolimits_{\argmin F} H\} = \mathcal{O}(k^{-2 + \delta}) \quad \text{as} \  k\to +\infty,
\end{equation}
which was reported in \cite{mst24}. Alternatively, if $\delta \in (0, 2)$ and $\delta\neq 1$, we can obtain \emph{simultaneous} (i.e., computed on the same sequence) ergodic convergence rates expressed in terms of
\begin{equation}
	x_{k}^{\text{best}} := \argmin \big\{ H(x_{k + 1}), H( \overline{x}_{k, \zeta})\big\}, \quad  \text{where} \quad \overline{x}_{k, \zeta} := \frac{1}{Z_{k, \delta}} \sum_{\ell = k_{0}}^{k} \zeta_{\ell, \delta} x_{\ell}, \quad Z_{k, \delta} :=\sum_{\ell = k_{0}}^{k} \zeta_{\ell, \delta}.
\end{equation}
Specifically, since $F(x_k)- \min F = \mathcal{O}(k^{-\delta})$, summing up from $k_0$ to $k$ inequality \eqref{eq:dissipation_second_order}, using Jensen's inequality, the convexity of $F$, and the growth $Z_{k, \delta}=\mathcal{O}(k^{-2 + \delta})$ as $k\to +\infty$, we obtain\footnote{Note that, strictly speaking, the bound on $H(x^k)-\min\nolimits_{\argmin F} H$ is not a \emph{convergence rate}, but only a \emph{one-sided estimate}. Here, we use a terminology from \cite{mst24}. We can provide a two-sided estimate, i.e., a convergence rate, only in the setting of Theorem \ref{thm:second_order_main}.}
\begin{equation}\label{eq:discussion_rates_3}
	F(x_{k}^{\text{best}}) - \min F = \mathcal{O}(k^{-\delta}) \quad \text{and} \quad H(x_{k}^{\text{best}}) - \min\nolimits_{\argmin F} H \leq \mathcal{O}(k^{-2+\delta})\quad \text{as} \ k\to +\infty.
\end{equation}
Furthermore, if $\delta=1$---which is the case considered in \cite[Theorem 3]{mst24}---the convergence rate for the inner function is weaker by a logarithmic factor than in \eqref{eq:discussion_rates_3}, $F(x_{k}^{\text{best}}) - F(x^{*}) = \mathcal{O}(k^{-1}\ln(k))$ as $k \to +\infty$, and $(\zeta_{\ell, \delta})_{k \geq 1}$ is bounded and bounded away from zero, so $\overline x_{k, \zeta}$ is actually closer to the usual ergodic sequence. It follows from \eqref{eq:discussion_rates_3}, that if $(x_{k}^{\text{best}})_{k \geq 1}$ is bounded (e.g., if $F$ or $H$ are coercive\footnote{In \cite{ms23}, the coercivity of \( H \) is assumed to ensure the boundedness of \( (x_k^{\text{best}})_{k \geq 1} \). Here, we show that, alternatively, it suffices to assume coercivity of \( F \) instead.}, once again from \eqref{eq:discussion_rates_3}), weak cluster points of $(x_{k}^{\text{best}})_{k \geq 1}$ lie in the solution set of \eqref{eq:bilevel_problem}. Under this assumption, if \eqref{eq:bilevel_problem} admits a unique solution, the whole sequence $(x_{k}^{\text{best}})_{k \geq 1}$ converges weakly to it. Alternatively, without assuming uniqueness of solutions but $\dim \cH < +\infty$, we immediately obtain that the distance from $(x^k)_{k \geq 0}$ to the solutions set of the bilevel problem  \eqref{eq:bilevel_problem} converges to zero as $k\to +\infty$, which is in line with \cite{ms23, mst24, ltvp25}. However, in infinite dimensions, no iterates convergence can be derived with this approach.

\paragraph{Weakening the H{\"o}lderian growth condition.} When $F$ is coercive, we can assume the H{\"o}lderian growth condition, i.e., Assumption \ref{ass:holderian_growth}, to hold only \emph{locally}. Indeed, suppose that, for some $R>0$, 
\begin{equation}
	\tau \rho^{-1}\dist(x, \argmin F)^{\rho} \leq F(x) - \min F \quad \text{for all} \  x \in C_R :=\argmin F + B_R(0).
\end{equation}
Using \eqref{eq:discussion_rates_1}, it is easy to see that $(x_k)_{k \geq 0}$ eventually lies in the bounded set $C_R$. Therefore, up to replacing the non-smooth part of $F$ by setting $\hat{f}_{R} := \hat{f} + \iota_{C_R}$, which does not change the algorithm, the same analysis assuming Assumption \ref{ass:holderian_growth} globally on $\hat f_R$ yields Point 3 in Theorem \ref{thm:second_order_main}. A similar remark can be done for Algorithm \ref{Bilevel Proximal Gradient algorithm}.

\paragraph{General regularization sequence $(\varepsilon_k)_{k \in \N}$.} In Lemma \ref{lem:dissipation_first_order} and Lemma \ref{lem:dissipation_second_order} we have intentionally made explicit the dependence on the regularization sequence $(\varepsilon_k)_{k \geq 0}$ and indeed the proof does not depend on the specific definition of $\varepsilon_k$ in \eqref{eq:def_epsilon}. Following the same arguments of Theorem \ref{thm: first order algorithm} and Theorem \ref{thm:second_order_main} we can thus extend our analysis to more general parameter sequences, provided that $\zeta_{k, \delta}\geq 0$ for all $k \geq 1$. Let us illustrate one of such extensions for Algorithm \ref{alg:first_order}.

Consider \eqref{eq:dissipation_first_order} and denote by $\zeta_{k, \varepsilon}$ the term $\zeta_{k, \delta}$ to emphasize that the term depends on $(\varepsilon_k)_{k \geq 0}$, which is not necessarily of polynomial decrease, but, rather, the discrete counterpart to $(\lambda - 1)\varepsilon(t) - t\dot \varepsilon(t)$, see $\zeta(t)$ in Appendix \ref{sec:continuous_time_first_order}. Then, for all $k \geq k_0$, applying the same argument as in \eqref{eq: Attouch-Czarnercki implies the integrability condition of the lemma}, we obtain, setting $\bar \zeta_{k, \varepsilon}:=\frac{\zeta_{k, \varepsilon}}{s(\lambda-1)}$, 
\begin{equation}\label{eq:general_epsilon_1}
	\begin{aligned}
		E_{k + 1}^{\lambda} - E_{k}^{\lambda} &\leq -s(\lambda - 1) \big[\bar \zeta_{k, \varepsilon} (H(x_{k}) - H(x^{*})) - (F(x_{k}) - F(x^{*})) \big]\\
		& \leq s(\lambda - 1)\big[ (F - \min F)^*(\bar \zeta_{k, \varepsilon} p^*) -  \sigma_C(\bar \zeta_{k, \varepsilon} p^*)\big].
	\end{aligned}
\end{equation}
In this context, the Attouch--Czarnecki condition, i.e., Assumption \ref{ass:attouch_czarnecki}, would now require the summability of the right-hand side of \eqref{eq:general_epsilon_1}. This yields, after suitably bounding $E_k^\lambda$ from below, that the limit  $\lim_{k\to+\infty}E_k^{\lambda}$ exists, and in particular, for some constant $C>0$,
\begin{equation}\label{eq:general_epsilon_2}
	t_k (F(x_k) - \min F) \leq C + t_k \varepsilon_k \min\nolimits_{\argmin F} H, \quad \text{i.e.,} \quad F(x_k) - \min F = \mathcal{O}(k^{-1} + \varepsilon_k) \quad \text{as} \ k \to+\infty.
\end{equation}
Plugging $\varepsilon_k$ as in \eqref{eq:def_epsilon}, \eqref{eq:general_epsilon_2} is in line with Point 2 of Theorem \ref{thm: first order algorithm}. The little-$o$ rate necessitates a more careful analysis, which we believe can be performed using the same techniques of Theorem \ref{thm: first order algorithm}.

\section{Numerical Experiments}\label{sec:numerical_experiments}

In this section, we present our numerical experiments, which are performed in Python on a 12thGen.~Intel(R) Core(TM) i7–1255U, 1.70–4.70 GHz laptop with 16 Gb of RAM and are available for reproducibility at \href{https://github.com/echnen/fast-bilevel-methods}{https://github.com/echnen/fast-bilevel-methods}.

\subsection{Considered algorithms}\label{sec:num_compared_algorithms}

To begin, we briefly present the numerical algorithms utilized for comparison to test the performance of Algorithms \ref{alg:first_order} and \ref{alg:second_order}. 
To the best of our knowledge, these are the only approaches that can cope with the generality of \eqref{eq:bilevel_problem}, so we refrain from discussing other variants requiring, e.g., strong convexity of the outer function.

\paragraph{Fast Bi-level Proximal Gradient.} The \emph{Fast Bi-level Proximal Gradient method} (FBi-PG) was recently introduced by Merchav, Sabach and Teboulle in \cite{mst24}. It can be understood a special case of Algorithm \ref{alg:second_order} with $\gamma = \alpha - 1$, $c = 1$, and $\beta=\gamma$. Specifically, although the method is introduced in \cite[(2.4)--(2.6)]{mst24} with $t_k$ picked as in the original Nesterov's method, the convergence results are only stated with $t_k:=\frac{k + a}{a}$ for some $a \geq 2$. This choice yields a momentum parameter $\alpha_k:=t_k^{-1}(t_{k-1}-1)$ that can be obtained from our general choice in Algorithm \ref{alg:second_order} by setting $\gamma =\alpha-1$ and $a:=\alpha-1$. Being a special case of Algorithm \ref{alg:second_order}, its convergence properties are specified in Theorem \ref{thm:second_order_main}. 

\paragraph{Static Bilevel Method.} The \emph{Static Bilevel Method} (staBiM) was introduced by Latafat, Themelis, Villa and Patrinos in \cite{ltvp25}. While the method is designed for an arbitrary slowly vanishing, i.e., not summable, regularization sequence $(\varepsilon_k)_{k \in \N}$, in these experiments we always consider $\varepsilon_k$ as in \eqref{eq:def_epsilon} with $\delta \in (0, 1)$. For a fixed $\tilde \theta \in (0, 1)$ and $\eta_0>0$, the method iterates for all $k\geq 0$
\begin{equation}
        x_{k+1} := \prox_{\theta_{k+1}(\hat f + \varepsilon_k \hat h)}\big( x_k - \theta_{k+1}(\nabla f(x_k) + \varepsilon_k\nabla h(x_k)) \big),
\end{equation}
where $(\theta_k)_{k \geq 1}$ is the sequence defined by $\theta_{k+1}:=\frac{\tilde \theta}{\eta_{k+1}L_{\nabla h} + L_{\nabla f}}$, where $\eta_{k+1}$ can be picked arbitrarily in $\big[\frac{3}{4}\eta_k, \eta_k\big]$ for all $k\geq 0$. Observe the similarity with Algorithm \ref{alg:first_order}: While allowing for general regularization sequences $(\varepsilon_k)_{k \geq 0}$, and originally designed in the more general context of locally Lipschitz gradients, it requires a non-stationary step size $\theta_k$ that, picking $\eta_{k+1}=\frac{3}{4}\eta_k$ for all $k \in \N$, asymptotically behaves like $\theta_k \simeq \frac{\tilde \theta}{L_{\nabla f}}$ with $\tilde \theta \in (0, 1)$. In contrast, our analysis suggests that indeed, we can choose i) stationary and ii) twice larger step sizes as $\theta_k=\frac{\tilde \theta}{L_{\nabla f}}$ and $\tilde \theta \in (0, 2)$ for all $k \geq 1$.

\paragraph{Bi-Sub-Gradient method.} The \emph{Bi-Sub-Gradient method} in its two different variants was introduced by Merchav and Sabach in \cite{ms23}. In this paper, we focus on the so-called Variant II, (Bi-SG-II), which, for $\varepsilon_k := \frac{c}{(k+ 1)^{\delta}}$, $\delta \in (\frac{1}{2}, 1]$, $c \leq \min\{\frac{1}{L_{\nabla h}}, 1\}$, and $\theta\leq \frac{1}{L_{\nabla f}}$, iterates for all $k\geq 0$
\begin{equation}
    \left\{\begin{aligned}
        &y_{k+1}:=\prox_{\theta \hat f}(x_k - \theta \nabla f(x_k)), \\
        &x_{k+1}:=\prox_{\theta \varepsilon_k \hat h}(y_{k+1} - \theta\varepsilon_k\nabla h(y_{k+1})).
    \end{aligned}\right.
\end{equation}
The method alternates a proximal-gradient step in the inner problem (with step size $\theta$) and in the outer problem (with step size $\theta\varepsilon_k$). The benefit of this formulation lies in the fact that it does not require computing the proximity operator of $\theta\hat f + \theta\varepsilon_k \hat h$, which for staBiM and Algorithm \ref{alg:first_order} requires a product-space trick as suggested in \cite{ds23}.

\subsection{Linear inverse problem with simulated data}\label{sec:num_nem}

To test the dependency of the proposed method on the algorithm's parameter, we consider the following bilevel problem
\begin{equation}\label{eq:num_bilevel_nemirovsky}
    \begin{aligned}
        \min_{x \in \R^d} \ \|x - \hat x\|_1 \quad \text{subject to} \ x \in \argmin_{z \in \R^d} f(z)\,, 
    \end{aligned}
\end{equation}
where $f: \R^d \to \R$ is the quadratic function defined, for $1<J<d$, by
\begin{equation}\label{eq:nemirowsky_function}
    f(x) := \frac{1}{2}(x_1 - 1)^2 + \frac{1}{2}\sum_{j = 2}^J(x_{j-1}-x_j)^2, \quad \text{for} \ x \in \R^d,
\end{equation}
and $\hat x :=(50, \dots, 50)$. The inner function \eqref{eq:nemirowsky_function} is a standard benchmark to test first-order methods, usually revealing their sublinear rates of convergence \cite{ny83}. In particular, it is not strongly convex, but satisfies Assumption \ref{ass:holderian_growth} with $\rho = 2$. Since $J<d$, the set of minimizers of $f$ is $\argmin f =\{x \in \R^d \ : \ x_1=\dots = x_J = 1\}$ and thus the unique solution to \eqref{eq:num_bilevel_nemirovsky} is $x^*:=(1,\dots, 1, 50, \dots, 50)$.

In our first experiment, we compare the performance of Algorithms \ref{alg:first_order} and \ref{alg:second_order} with $\beta = 10$, $\gamma=20$, $c = 10$, and $\alpha = 4$ against the methods described in Subsection \ref{sec:num_compared_algorithms}. For fair comparison, for each of them we set $95\%$ of the largest possible step size, and consider the same initial point $x^0:=0 \in \R^d$. Then, we let the algorithms run for $10^4$ iterations with $20$ different choices of $\delta \in (1, 2)$, sampled randomly---for Algorithm \ref{alg:first_order}, staBiM and Bi-SG-II we use $\delta/2$ instead of $\delta$ to ensure convergence. We measure inner and outer residual values and show the results in Figure \ref{fig:nem}.

Although the objective of this paper is not to compete with existing algorithms, we can observe that the improved flexibility of Algorithm \ref{alg:second_order} plays a crucial role in performance, yielding results that clearly surpass those of FBi-PG. In this example, specifically from Figure \ref{fig:nem_obj_inner}, we can also observe that the first-order methods Algorithm \ref{alg:first_order}, staBiM and Bi-SG-II may indeed only achieve the rate $o(k^{-1})$ on the inner objective function decrease. On the other hand, the two second-order methods employed, i.e., Algorithm \ref{alg:second_order} and FBi-PG, show a much faster convergence behavior. Further, Figure \ref{fig:nem_delta} confirms the inherent trade-off between inner and outer function performance as $\delta$ varies, cf.~Section \ref{sec:discussion}.

\begin{figure}[t]
    \centering
    \begin{subfigure}{0.33\linewidth}
    \includegraphics[width=\linewidth]{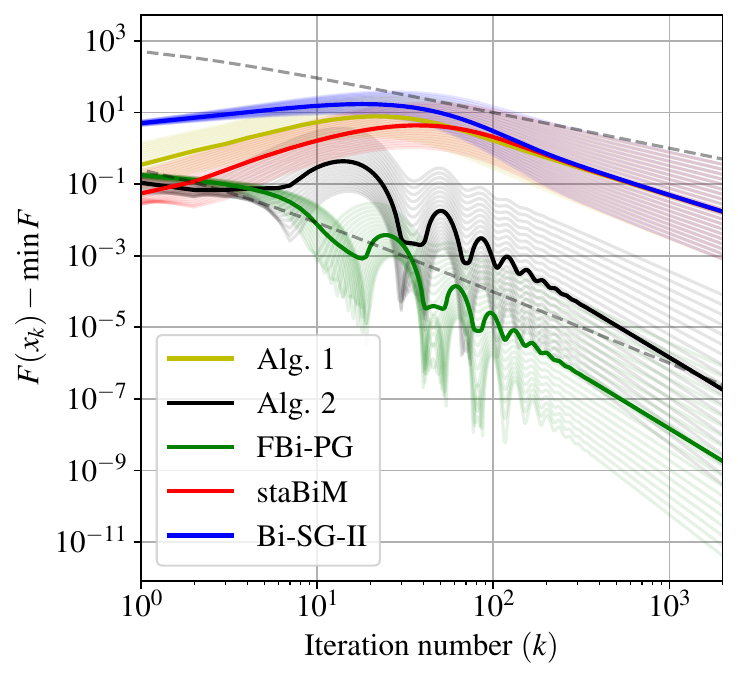}
    \caption{Inner objective residual.}
    \label{fig:nem_obj_inner}    
    \end{subfigure}
    \begin{subfigure}{0.32\linewidth}
    \includegraphics[width=\linewidth]{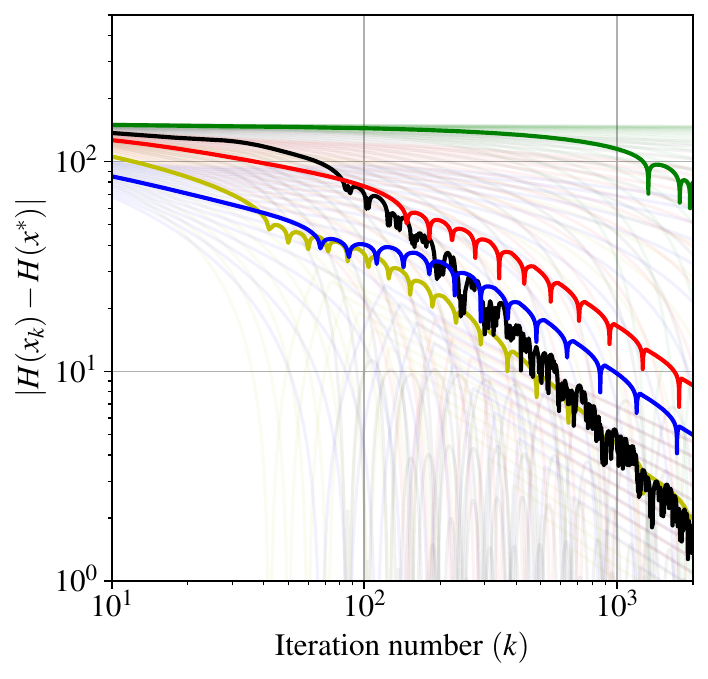}
    \caption{Outer objective residual.}
    \label{fig:nem_obj_outer}    
    \end{subfigure}
    \begin{subfigure}{0.32\linewidth}
    \includegraphics[width=\linewidth]{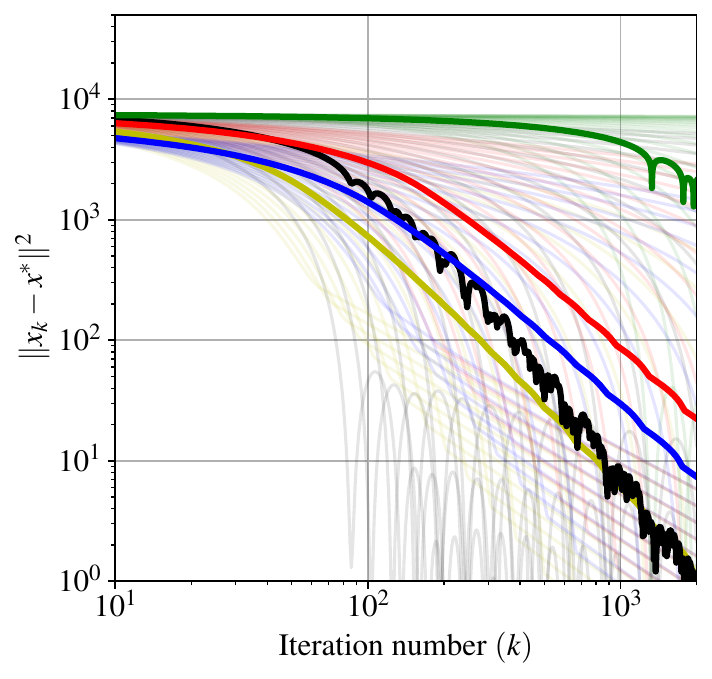}
    \caption{Distance to solution.}
    \label{fig:nem_obj_dist}    
    \end{subfigure}
    \caption{Result of the experiment in Subsection \ref{sec:num_nem}: Comparing inner and outer function residuals as well as distance to solution for the methods listed in Subsection \ref{sec:num_compared_algorithms}.}
    \label{fig:nem}
\end{figure}%
\begin{figure}[t]
    \centering
    \begin{subfigure}{0.29\linewidth}
    \includegraphics[width=\linewidth]{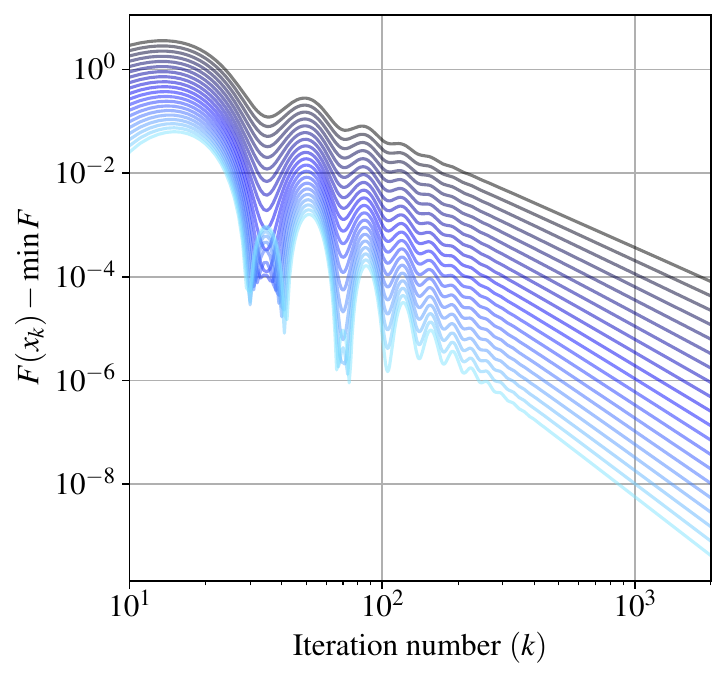}
    \caption{Inner objective residual}
    \label{fig:nem_obj_inner_delta}    
    \end{subfigure}\hspace{0.5cm}
    \begin{subfigure}{0.35\linewidth}
    \includegraphics[width=\linewidth]{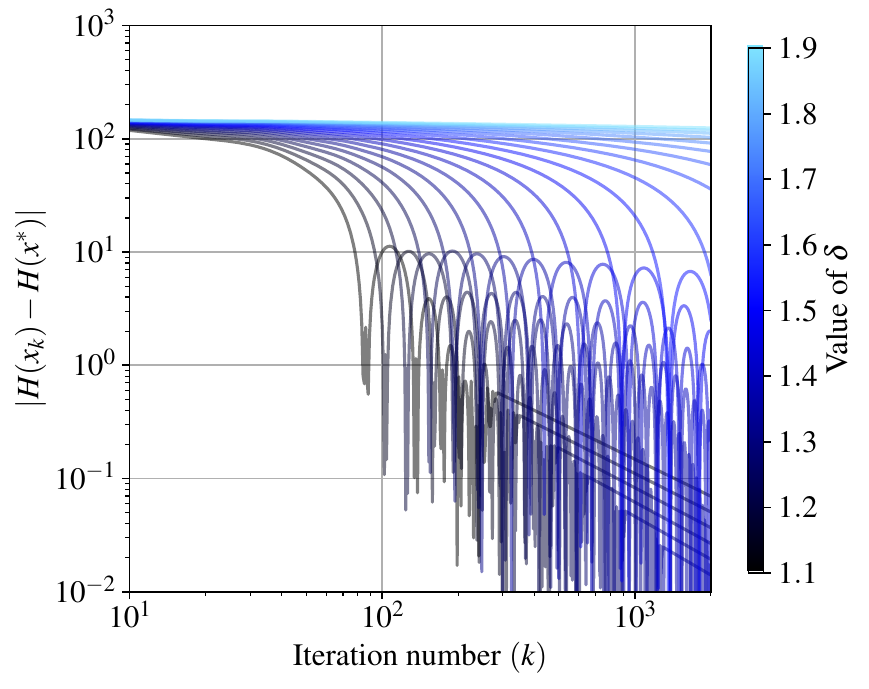}
    \caption{Outer objective residual}
    \label{fig:nem_obj_outer_delta}    
    \end{subfigure}
    \caption{Results of experiment in Section \ref{sec:num_nem}: Comparing residual decrease for Algorithm \ref{alg:second_order} as a function of the iteration number emphasizing the choice of $\delta \in (1, 2)$.}
    \label{fig:nem_delta}
\end{figure}

\subsection{Logistic regression}\label{sec:num_logistic}

In this experiment, we consider a copy of the UCI ML Breast Cancer Wisconsin (Diagnostic) dataset \cite{dataset} freely available online \cite{scikit11}. It contains $n=455$ instances with $30$ features and a label $y_i$ ($i=1,\dots, 455$) denoting weather the lesion is benign $y_i=1$ or malign $y_i=0$. We aim training a quadratic classifier, which for each sample $a \in \R^{30}$ yields $1$ if $\mathfrak{s}(\hat a^T x)>\frac{1}{2}$ and $0$ else, where $\hat a \in \R^{d}$ is the lifted feature vector of $a$ that now belongs to a $d=5456$ dimensional space, and $\mathfrak{s}(t):=(1 + \exp(-t))^{-1}$ is the standard sigmoid function. Due to overparametrization, among the optimal classifiers, we seek the one that minimizes the $\ell^1$ norm by solving the following bilevel problem
\begin{equation}
    \min_{x \in \R^d} \ \|x\|_1 \quad \text{subject to} \ x \in \argmin_{z\in \R^d} \ \frac{1}{n} \sum_{i=1}^n \bigg( y_i \log(\mathfrak{s}(\hat a_i^T z))  + (1-y_i)\log(1-\mathfrak{s}(\hat a_i^T z)) \bigg).
\end{equation}

To train the model, we run the three numerical algorithms presented in Subsection \ref{sec:num_compared_algorithms} and Algorithms \ref{alg:first_order} and \ref{alg:second_order} for $5\cdot 10^4$ iterations with the $\delta=1.9$ for the second-order methods, and $\delta/2 = 0.95$ for the first-order ones. Specifically, for Algorithm \ref{alg:second_order} we set $\gamma=\beta=1$, $\alpha=4$, $c=10^2$, $s=\frac{0.95}{L_{\nabla f}}$. For FBi-PG, $\alpha=4$, and $s=\frac{0.95}{L_{\nabla f}}$. For Algorithm \ref{alg:first_order} we set $c=10^2$, $s=\frac{1.95}{L_{\nabla f}}$ and $\beta=\gamma=1$, while for staBiM $\gamma =1$, $c=10^2$ and for Bi-SG-II, $s=\frac{1}{L_{\nabla f}}$, and $c=10^2$. Along the iterations, we measure the inner and outer objective values and the norm of the difference of consecutive iterates and show the results in Figure \ref{fig:logistic}. Since the optimal solution cannot be reached within a reasonable time, for the outer function we only show its function values, while for the inner we show the residual with the minimum value attained.

From Figure \ref{fig:logistic_obj_inner}, we observe that in this example, the convergence rate perfectly matches the one predicted by the theory, see Section \ref{sec:discussion}. The fast schemes (i.e., Algorithm \ref{alg:second_order} and FBi-PG) exhibit a convergence rate of \( o(k^{-\delta}) \), outperforming the non-accelerated ones (i.e., Algorithm \ref{alg:first_order}, staBiM and Bi-SG-II) that only exhibit a convergence rate of \( o(k^{-\delta/2}) \). Meanwhile, Figure \ref{fig:logistic_obj_outer} indicates that at least one of the considered methods may still be away from the optimal solution, and Figure \ref{fig:logistic_red} shows that the distance between consecutive iterates decreases similarly across all four algorithms.

\begin{figure}[t]
    \centering
    \begin{subfigure}{0.32\linewidth}
    \includegraphics[width=\linewidth]{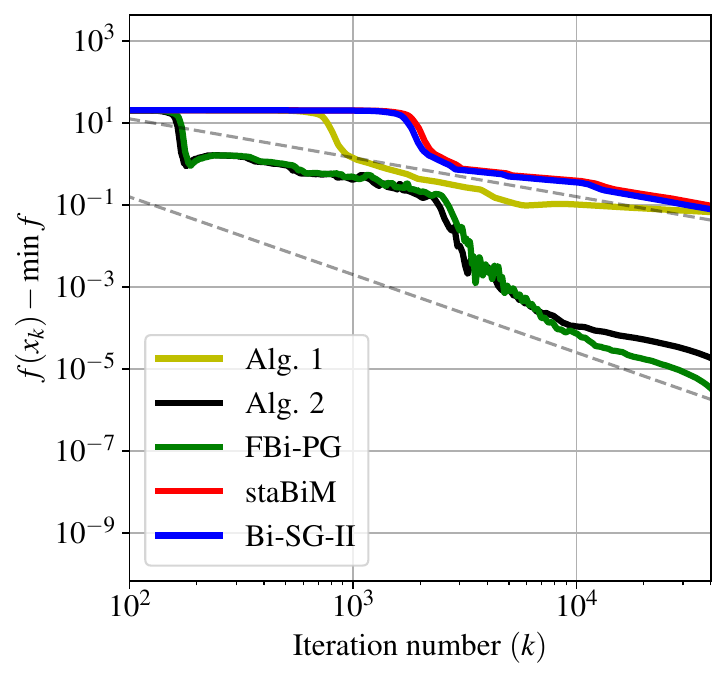}
    \caption{Inner function residual.}
    \label{fig:logistic_obj_inner}    
    \end{subfigure}
    \begin{subfigure}{0.31\linewidth}
    \includegraphics[width=\linewidth]{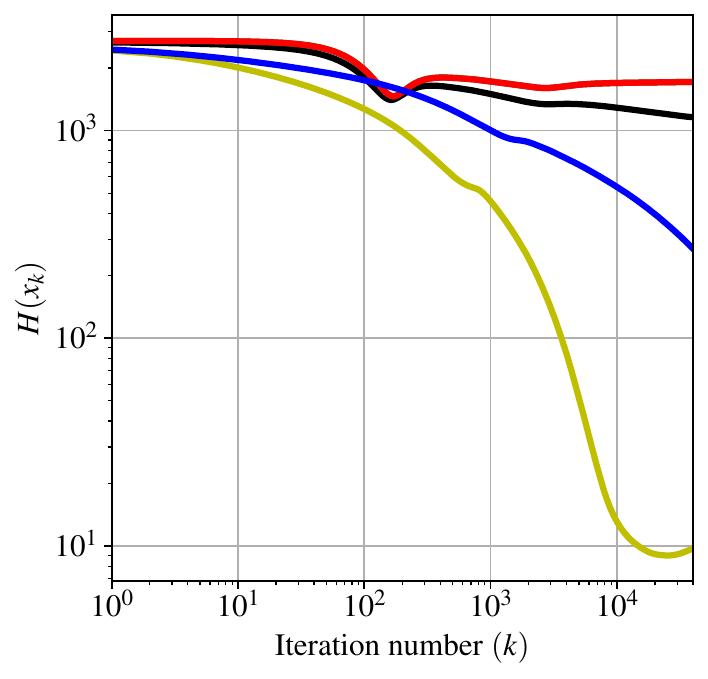}
    \caption{Outer function.}
    \label{fig:logistic_obj_outer}    
    \end{subfigure}
    \begin{subfigure}{0.32\linewidth}
    \includegraphics[width=\linewidth]{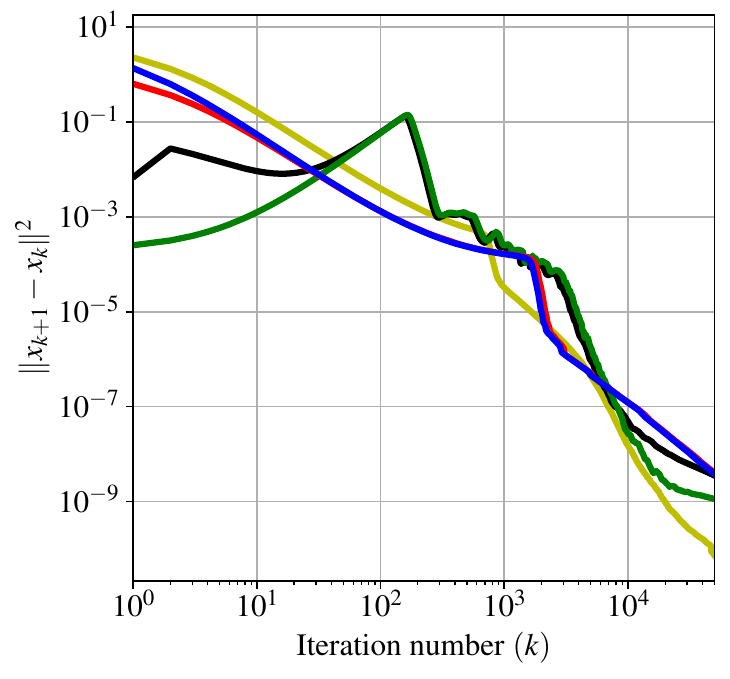}
    \caption{Distance of consecutive iterates.}
    \label{fig:logistic_red}    
    \end{subfigure}
    \caption{Results of experiment in Subsection \ref{sec:num_logistic}. The dotted lines in Figure \ref{fig:logistic_obj_inner} represent the rates $\mathcal{O}(k^{-\delta / 2})$ and $\mathcal{O}(k^{-\delta})$, with $\delta=1.9$.}
    \label{fig:logistic}
\end{figure}

\appendix

\section{Auxiliary results}

The following straightforward result encodes an important inequality that is contained, e.g., in \cite{ms23}, and is used repeatedly in our paper.
\begin{lem}\label{lem: Hölderian error bound implies the second condition of the lemma}
	Suppose Assumptions \ref{ass:holderian_growth} and \ref{ass:qualification} hold. Then for any solution $x^*\in \cH$ to \eqref{eq:bilevel_problem}, and $p^* \in N_{\argmin F}(x^*)$ such that $-p^* \in \partial H(x^*)$ it holds
	\begin{equation}\label{eq: Hölderian error bound implies the second condition of the lemma}
		H(x^{*}) - H(x)\leq  \|p^{*}\| \left(\frac{F(x) - F(x^{*})}{\tau \rho^{-1}}\right)^{\frac{1}{\rho}} \quad \text{for all} \ x \in \cH.
	\end{equation} 
\end{lem}
\begin{proof}
Using the definitions of the convex subdifferential and of the normal cone, for all $x\in\cH$ it holds
\begin{align}
	H(x^{*}) - H(x) &\leq -\langle p^{*}, x^{*} - x\rangle \leq -\bigl\langle p^{*}, P_{\argmin F}(x) - x\bigr\rangle \nonumber\\
	&\leq \| p^{*}\| \bigl\| P_{\argmin F}(x) - x\bigr\| = \| p^{*}\|  \dist(x, \argmin F) \leq \| p^{*}\| \left(\frac{F(x) - F(x^{*})}{\tau \rho^{-1}}\right)^{\frac{1}{\rho}},
\end{align}
which yields the claim.
\end{proof}

For the proof of the following classical lemma, see, e.g., \cite{sbc16}.
\begin{lem}\label{lem: proximal step}
Let $\hat{\Phi} : \cH \to \R \cup \{+\infty\}$ be a proper, convex and lower semicontinuous function and ${\Phi} : \cH \to \R$ be a $L_{\nabla \Phi}$-smooth function. For $s > 0$ and $y\in \mathcal{H}$, define 
	\begin{equation}\label{eq: definition of G_s}
		G(y) := \frac{1}{s} \Bigl[y - \prox_{s \hat{\Phi}}( y - s \nabla \Phi(y))\Bigr].
	\end{equation}
	Set $\Psi := \Phi + \hat{\Phi}$. Then, for all $x, y \in \mathcal{H}$ it holds
	\begin{equation}
		\Psi(y - sG(y)) \leq \Psi(x) - \frac{s}{2}(2 - sL_{\nabla \Phi}) \| G(y)\|^{2} + \langle G(y), y - x\rangle - \frac{1}{2L_{\nabla \Phi}} \| \nabla\Phi(x) - \Phi(y)\|^{2}.
	\end{equation}
\end{lem}

The following result follow from simple integral estimates. 
\begin{lem}\label{lem: sums}
Let $r\in \R$. Then, for all $k_{0} \geq 2$ and all $k\geq k_{0}$ it holds
    \begin{itemize}
        \item If $r < 0$, then $\sum_{\ell = k_{0}}^{k} \frac{1}{\ell^{r}} \in \left[\frac{1}{1 - r} \Bigl(k^{1 - r} - (k_{0} - 1)^{1 - r}\Bigr), \frac{1}{1 - r} \Big( (k + 1)^{1 - r} - k_{0}^{1 - r}\Bigr)\right].$
        \item If $r = 1$, then $ \sum_{\ell = k_{0}}^{k} \frac{1}{\ell} \in \Bigl[ \ln(k + 1) - \ln(k_{0}), \ln(k) - \ln(k_{0} - 1)\Bigr].$
        \item If $r\geq 0$ and $r\neq 1$, then $\sum_{\ell = k_{0}}^{k} \frac{1}{\ell^{r}} \in \left[ \frac{1}{1 - r} \left( \frac{1}{(k + 1)^{r - 1}} - \frac{1}{k_{0}^{r - 1}}\right), \frac{1}{1 - r} \left( \frac{1}{k^{r - 1}} - \frac{1}{(k_{0} -  1)^{r - 1}}\right)\right].$
        \item In particular, if $1 - r > 0$, then $\sum_{k = k_0}^{\ell} \frac{1}{\ell^{r}} \leq \frac{1}{1 - r} \Bigl[(k + 1)^{1 - r} - (k_{0} - 1)^{1 - r}\Bigr].$
    \end{itemize}
\end{lem}

A special case (i.e., when $t_{k} = k$) of the following lemma can be found in the literature, for example, in \cite[Lemma A.5]{fastOGDA}. However, that version requires the sequence $(\varphi_{k})_{k \geq 0}$ to be bounded, which we cannot assume. This is important, since when applying this lemma we do not know beforehand if $(\varphi_{k})_{k \geq 0}$ is bounded or not.  
\begin{lem}\label{lem: for the existence of lim|xk-x*|}
    Let $(\varphi_{k})_{k \geq 0}, (t_{k})_{k \geq 0}$ be sequences of real and positive numbers respectively, and such that $\sum_{k = 0}^{+\infty} \frac{1}{t_{k}} = +\infty$. Then,
    \begin{equation}
    \lim_{k\to +\infty} \left( \varphi_{k + 1} + t_k (\varphi_{k + 1} - \varphi_{k})\right) = L \in\R  \quad \implies \quad \lim_{k\to +\infty} \varphi_{k} = L. 
    \end{equation} 
\end{lem}
\begin{proof}
For all $k \geq 0$, denote $\ell_k:= \varphi_{k+1} + t_k (\varphi_{k+1}-\varphi_k)$ and suppose that $\ell_k\to L$ as $k \to +\infty$. Consider the positive sequence $(\mu_k)_{k \geq 0}$ defined recursively as
	\begin{equation}\label{eq:proof_phi_converges_2}
		\mu_0 := \frac{1}{t_0} \quad \mbox{and} \quad \mu_{k+1} := \frac{\mu_k(1+t_k)}{t_{k+1}} \quad \text{for all} \ k \geq 0.
	\end{equation}
We denote by $P_k:= \mu_k t_k$ for all $k \geq 0$. Then, $P_0 = 1$ and for all $k \geq 0$
	\begin{equation}\label{eq:proof_phi_converges_3}
		P_k >0, \quad \sum_{i=0}^k \mu_i = P_{k+1}-1 \quad \text{and} \quad
		P_{k+1} = \prod_{j=0}^k \bigg(\frac{1}{t_j} + 1\bigg) \geq 1 + \sum_{j = 0}^k \frac{1}{t_j},  \quad \text{thus} \quad P_k \to +\infty,
	\end{equation}
	where the first three statements follow from \eqref{eq:proof_phi_converges_2}, and the fourth statement from the summability assumption on $(t_k)_{k \geq 0}$. Multiplying $\ell_k$ by $\mu_k$ yields
	\begin{equation}\label{eq:proof_phi_converges_1}
		\begin{aligned}
			\mu_{k+1}t_{k+1}\varphi_{k+1} - \mu_k t_k \varphi_{k} &= (\mu_{k+1}t_{k+1}- \mu_{k}t_k)\varphi_{k+1} + \mu_k t_k (\varphi_{k+1}-\varphi_k)\\
			& =\mu_k \varphi_{k+1} + \mu_k t_k (\varphi_{k+1}-\varphi_k) = \mu_k \ell_k.
		\end{aligned}
	\end{equation}
	Therefore, by summing \eqref{eq:proof_phi_converges_1} from $0$ to $k$, we obtain for all $k \geq 0$
	\begin{equation}
		\varphi_{k+1} = \frac{\varphi_0}{P_{k+1}} + \frac{1}{P_{k+1}} \sum_{i = 0}^k \mu_i \ell_i = \frac{\varphi_0}{P_{k+1}} + \frac{\sum_{i = 0}^k \mu_i \ell_i}{\sum_{i=0}^k \mu_i + 1}.
	\end{equation}
	In light of \eqref{eq:proof_phi_converges_3}, applying the Stolz--Ces\'aro Theorem leads to $\varphi_{k}\to L$ as $k \to +\infty$. 
\end{proof}

Opial's Lemma \cite{Opial} is used in the proof of the weak convergence of the sequences generated by our algorithms.
\begin{lem}[Opial's Lemma, discrete version]\label{lem: discrete Opial}
	Let $S$ be a nonempty subset of $\mathcal{H}$ and let $(x_{k})_{k \geq 0}$ be a sequence in $\mathcal{H}$. Assume that:
	\begin{enumerate}
		\item For every $x^{*} \in S$, $\lim_{k\to +\infty} \| x_{k} - x^{*}\|$ exists;
		\item Every weak sequential cluster point of $(x_{k})_{k \geq 0}$ belongs to $S$.
	\end{enumerate}
	Then, $(x^k)_{k \geq 0}$ converges weakly to an element in $S$.
\end{lem}

\begin{lem}\label{lem: to improve the rate for Fk, discrete}
Assume that $t_{k} := \theta(k + \gamma)$, with $\theta > 0$ and $\gamma \geq 0$, and that $(F_{k})_{k \geq 0}$ is a sequence of nonnegative numbers such that, for some $k_0 \geq 0$,
\begin{equation}
F_{k} \leq C_{E} t_k^{-\eta} + C_{0, F} t_{k}^{- \delta} F_{k}^{\frac{1}{\rho}} \quad \text{for all} \ k \geq k_0, \quad \mbox{and} \quad F_{k} = \mathcal{O}(k^{- \delta}) \quad \text{as} \ k\to +\infty,
	\end{equation}
where $\rho\in (1, 2]$, $\eta\in\{1, 2\}$, $ \delta\in \big(\frac{\eta}{\rho^\star}, \eta \big)$ (where $\rho^{\star}$ is set as in \eqref{eq:def_rho_star}), and $C_{E}$ and $C_{0, F}$ are positive constants. Then, it holds 
	\begin{equation}
		F_{k} = \mathcal{O}( k^{-\eta}) \quad \text{as} \ k\to +\infty.
	\end{equation}
\end{lem}
\begin{proof}
	Since $F_k=O(k^{- \delta})$ as $k\to+\infty$, there is a positive constant $C_{1, F}$ such that for all $k\geq k_{0}$ it holds
	\begin{equation}
	F_k \leq C_{E}t_k^{-\eta} + C_{0, F} t_k^{- \delta - \frac{1}{\rho} \delta}(t_k^{\delta}F_k)^{\frac{1}{\rho}} \leq C_{E}t_k^{-\eta} + C_{1, F} t_k^{- \delta - \frac{1}{\rho} \delta} .
	\end{equation} 
	If $\delta_1 := \delta + \frac{\delta}{\rho} \geq \eta$, we stop. Otherwise, the rate has improved to $F_k = \mathcal{O}(k^{-\delta_1})$ as $k\to+\infty$. We may repeat this first step again, but this time plugging $t_k^{\delta_1}$ instead of $t_k^{\delta}$, possibly improving the rate to $F_{k} = \mathcal{O}(k^{-\delta_{2}})$, where $\delta_{2} = \delta_{1} + \frac{\delta_{1}}{\rho} = \delta(1 + \rho^{-1} + \rho^{-2})$. We may repeat this argument $n$ times, as long as $\delta_{n+1}:=\delta_n + \frac{\delta_n}{\rho} < \eta$, in order to obtain a positive constant $C_{n, F}$ such that for all $k\geq k_{0}$
	\begin{equation}
		F_k \leq C_{E} t_k^{-\eta} +  C_{n, F} t_k^{- \delta_{n}}, \quad \text{where} \quad  \delta_{n} := \delta (1 + \rho^{-1} +\dots +\rho^{-n}). 
	\end{equation}
	Since $\delta_{n} \to \delta \rho^{\star} > \eta$ as $n\to +\infty$, we will stop at some $n_{0}$ such that $\delta_{n_{0} - 1} < \eta$ and $\delta_{n_{0}} \geq \eta$, which yields the claim.
\end{proof}

\begin{lem}\label{lem: general descent lemma for discrete case}
Let $\hat \delta >0$, $\eta \in \{1, 2\}$, and set $\delta := \eta \hat \delta$. Consider a sequence $(E_{k})_{k \geq 1}$ defined as 
\[
    E_{k} = t_{k}^{\eta} \Big(F_{k} + \varepsilon_{k - 1} H_{k}\Big) + V_{k} \quad \text{for all} \ k \geq 1,
\]
with $t_{k} = \theta(k + \gamma)$, $(\varepsilon_{k})_{k\geq 0}$ defined as in \eqref{eq:def_epsilon}, and $F_k \geq 0$, $H_{k} \geq -H_{*}$, where $H_* \geq 0$, and $V_k \geq 0$ for all $k \geq k_0,$ for some $k_0\geq1$. Suppose that
	\begin{equation}\label{eq: general inequality for E(k+1)-E(k), lemma, discrete_appendix}
		E_{k + 1} - E_{k} + \zeta_{k, \delta} H_{k} + g_{k} \leq -C_{2} k^{\eta - 1} F_{k} \quad \text{for all} \ k \geq k_0,
	\end{equation}
	 where $(g_k)_{k \geq 1}$ is a nonnegative sequence, $C_2 \geq 0$, and $(\zeta_{k, \delta})_{k \geq 1}$ satisfies, for $C_{1, \zeta} > 0$, $C_{2, \zeta} > 0$,
	\begin{equation}\label{eq:general descent lemma for discrete case, assumption on theta}
	C_{1, \zeta} k^{\eta-1} \varepsilon_{k - 1} \leq \zeta_{k, \delta} \leq C_{2, \zeta} k^{\eta-1} \varepsilon_k \quad \text{for all} \ k \geq k_0.
	\end{equation}
	Consider the following conditions, which will only be assumed if we explicitly say so. For all $k\geq k_0$: 
	\begin{align*}
		(\labterm{eq: Hölderian error bound condition lemma, discrete}{C1}) & \quad -H_{k} \leq C_{3} F_{k}^{\frac{1}{\rho}}, \ \text{for some constant $C_{3} \geq 0$ and $1 < \rho \leq 2$. Set $\rho^\star$ as in \eqref{eq:def_rho_star}}, \\
        (\labterm{eq: Attouch-Czarnecki condition lemma, discrete}{C2}) & \quad \text{For each $r > 0$,} \ -k^{\eta - 1} \big( F_{k} + r \varepsilon_k H_{k}\big) \leq I_{k}^{r}, \ \text{where $I_{k}^{r} \geq 0$ and $(I_{k}^{r})_{k \geq 1} \in \ell^{1}(\N)$}.
	\end{align*}
	Then, following statements are true:
	\begin{itemize}
		\item Case $\hat \delta > 1$: It holds $F_{k}=\mathcal{O} (k^{-\eta} )$ as $k \to +\infty$;
		\item Case $\hat \delta = 1$: If Assumption \eqref{eq: Attouch-Czarnecki condition lemma, discrete} holds and $C_{2} > 0$ in \eqref{eq: general inequality for E(k+1)-E(k), lemma, discrete_appendix}, then
		\begin{enumerate}
			\item\label{item:general_descent_lemma_1_lim} The limit $\lim_{k\to +\infty} E_{k}$ exists;
			\item\label{item:general_descent_lemma_1_F} $F_k = \mathcal{O} ( k^{-\eta} )$ as $k\to +\infty$;
			\item\label{item:general_descent_lemma_1_sum} The sequences $(k^{\eta - 1} F_{k})_{k \geq 1}$, $\left( k^{\eta - 1}\varepsilon_{k - 1} H_{k}\right)_{k \geq 1}$ and $(g_{k})_{k \geq 1}$ are summable.
		\end{enumerate}
		\item Case $\frac{1}{\rho^\star} < \hat \delta < 1$: If Assumptions \eqref{eq: Hölderian error bound condition lemma, discrete} and \eqref{eq: Attouch-Czarnecki condition lemma, discrete} hold and $C_{2} > 0$ in \eqref{eq: general inequality for E(k+1)-E(k), lemma, discrete_appendix}, then
		\begin{enumerate}
			\item\label{item:general_descent_lemma_0_F} $F_{k} = \mathcal{O}( k^{-\eta} )$ as $k\to+\infty$;
			\item\label{item:general_descent_lemma_0_lower_bound} It holds $-C_H k^{-\delta + \frac{\eta}{\rho^\star}}\leq t_{k}^{\eta} \varepsilon_{k - 1} H_{k}$ for some $C_H\geq 0$ and all $k \geq k_0$;
			\item\label{item:general_descent_lemma_0_lim} The limit $\lim_{k\to +\infty} E_{k}$ exists;
			\item\label{item:general_descent_lemma_0_sum} The sequences $ (k^{\eta - 1} F_{k})_{k\geq 1}$, $\left(k^{\eta - 1}\varepsilon_{k - 1} H_{k}\right)_{k\geq 1}$ and $ (g_{k})_{k\geq 1}$ are summable.
		\end{enumerate} 
	\end{itemize}
\end{lem}
\begin{proof}	First, in general, using $-H_{*} \leq H_{k}$, summing \eqref{eq: general inequality for E(k+1)-E(k), lemma, discrete_appendix} from $k_{0}$ to $k\geq k_{0}$ and dropping the nonnegative terms $V_{k}$ and $g_{k}$, we get 
	\begin{equation}\label{eq: for the rates for Fk and Hk, lemma, discrete}
		t_{k + 1}^{\eta} F_{k + 1} - \bigg(t_{k+1}^\eta\varepsilon_{k} + \sum_{\ell = k_0}^k \zeta_{\ell, \delta} \bigg)H_{*} \leq E_{k_{0}}.
	\end{equation}
	Note that according to \eqref{eq:general descent lemma for discrete case, assumption on theta} and Lemma \ref{lem: sums}, we have, for some $C_{3, \zeta}, C_{4, \zeta}>0$ large enough, for all $k \geq k_0$
	\begin{equation}\label{eq: for the rates for Fk 2, lemma, discrete}
		\sum_{\ell = k_{0}}^{k} \zeta_{\ell, \delta} \leq C_{2, \zeta} \sum_{\ell = k_{0}}^{k} \ell^{\eta-1}\varepsilon_\ell \leq C_{3, \zeta} \sum_{\ell = k_{0}}^{k}  \frac{1}{\ell^{\delta - \eta + 1}} \leq 
		\begin{cases}
			C_{4, \zeta} & \text{if $\delta > \eta$, \quad i.e., $\hat \delta >1$}, \\
			C_{4, \zeta} \ln(k) &\text{if $\delta = \eta$, \quad i.e., $\hat \delta =1$}, \\
			C_{4, \zeta} k^{\eta - \delta}  &\text{if $\delta < \eta$, \quad i.e., $\hat \delta <1$}.
		\end{cases}
	\end{equation}

	\textbf{Case $\hat \delta > 1$:} Since $t_{k+1}^{\eta}\varepsilon_{k} = \mathcal{O}(k^{-(\delta - \eta)})$ vanishes and $\sum_{\ell=k_0}^k\zeta_{\ell, \delta}$ is bounded, \eqref{eq: for the rates for Fk and Hk, lemma, discrete} yields $F_k = \mathcal{O}(k^{-\eta})$ as $k\to+\infty$, which proves the convergence rate in this case.
	
Before proceeding with the two remaining cases, let us collect a few facts. First, notice that when $0 < \hat \delta \leq 1$, we have $t_{k + 1}^{\eta}\varepsilon_{k} = \mathcal{O}(k^{\eta - \delta})$, thus from \eqref{eq: for the rates for Fk and Hk, lemma, discrete} and \eqref{eq: for the rates for Fk 2, lemma, discrete} we get, as $k\to+\infty$
	\begin{equation}\label{eq: pre rates for Fk, case 0<delta<=1, lemma, discrete}
		F_{k} = \mathcal{O}(\ln(k) k^{-\eta}), \quad \text{if} \ \delta = \eta, \quad \text{and} \quad F_k = \mathcal{O}(k^{-\delta}) \quad \text{if} \ \delta < \eta.
	\end{equation}
Furthermore, we need three inequalities which are a consequence of \eqref{eq: general inequality for E(k+1)-E(k), lemma, discrete_appendix} when $C_{2} > 0$. After dropping the nonnegative term $g_{k}$ when required, we come for all $k \geq k_0$ to 
	\begin{align}
		E_{k + 1} - E_{k} \leq E_{k + 1} - E_{k} + \frac{C_{2} k^{\eta - 1}}{2} F_{k} + g_{k} &\leq - \frac{C_{2} k^{\eta - 1}}{2} \left( F_{k} + \frac{2 \zeta_{k, \delta}}{C_{2} k^{\eta - 1}} H_{k}\right), \label{eq: inequality 1 for E(k+1)-Ek, lemma, discrete}\\
		E_{k + 1} - E_{k} + \frac{\zeta_{k, \delta}}{2} H_{k} &\leq -C_{2} k^{\eta - 1} \left( F_{k} + \frac{\zeta_{k, \delta}}{2 C_{2} k^{\eta - 1}}H_{k}\right), \label{eq: inequality 2 for E(k+1)-Ek, lemma, discrete}\\
		E_{k + 1} - E_{k} - \zeta_{k, \delta} H_{k} &\leq -C_{2} k^{\eta - 1} \left( F_{k} + \frac{2 \zeta_{k, \delta}}{C_{2} k^{\eta - 1}} H_{k}\right). \label{eq: inequality 3 for E(k+1)-Ek, lemma, discrete}
	\end{align}
Furthermore, since $F_{k} \geq 0$, for all $r > 0$ and all $k \geq k_0$ we have, using \eqref{eq:general descent lemma for discrete case, assumption on theta}
	\begin{align*}
		- k^{\eta - 1} \left( F_{k} +  r \zeta_{k, \delta} k^{1-\eta} H_{k}\right) \leq
		\begin{cases}
			0 &\text{if $H_{k} > 0$,} \\
			 - k^{\eta - 1} \left( F_{k} +  r C_{2, \zeta} \varepsilon_k  H_{k}\right) &  \text{if $H_{k} \leq 0$}.
		\end{cases}
	\end{align*}
Therefore, when \eqref{eq: Attouch-Czarnecki condition lemma, discrete} holds, for all $r>0$ there exists $(I_k^{rC_{2, \zeta}})_{k \in \N}\in\ell^1(\N)$ nonnegative such that
	\begin{equation}\label{eq: term with Fk, Hk, Thetak is upper bounded by a summable sequence, lemma, discrete}
		- k^{\eta - 1} \left( F_{k} + r\zeta_{k, \delta} k^{1 - \eta}  H_{k}\right) \leq I_{k}^{r C_{2, \zeta}} \quad \text{for all} \ k\geq k_0.
	\end{equation}
	First, we establish that $\lim_{k\to+\infty}E_k$ exists, and then come to the summability statements.
	
\textbf{Case $\hat \delta = 1$:} Suppose \eqref{eq: Attouch-Czarnecki condition lemma, discrete} holds. On the one hand, when $\hat \delta = 1$, we have $E_{k} \geq -cH_{*}$ for $k$ large enough. On the other hand, according to \eqref{eq: inequality 1 for E(k+1)-Ek, lemma, discrete} and \eqref{eq: term with Fk, Hk, Thetak is upper bounded by a summable sequence, lemma, discrete}, $(E_{k + 1} - E_{k})_{k \geq 1}$ is upper bounded by a summable sequence. Combining these facts yields the existence of the limit $\lim_{k\to +\infty} E_{k}$, i.e., Point \ref{item:general_descent_lemma_1_lim}. In particular, this also yields $F_{k} = \mathcal{O}( k^{-\eta})$ as $k\to +\infty$, i.e., Point \ref{item:general_descent_lemma_1_F}. 
	
\textbf{Case $\frac{1}{\rho^\star} < \hat \delta < 1$:} Suppose \eqref{eq: Hölderian error bound condition lemma, discrete} and \eqref{eq: Attouch-Czarnecki condition lemma, discrete} hold. Using \eqref{eq: inequality 1 for E(k+1)-Ek, lemma, discrete} and \eqref{eq: term with Fk, Hk, Thetak is upper bounded by a summable sequence, lemma, discrete} we get that $(E_{k + 1} - E_{k})_{k \geq 1}$ is upper bounded by a summable sequence. Therefore, $E_{k} \leq C_{E}$ for some constant $C_E$ and all $k \geq k_0$. In particular, using also \eqref{eq: Hölderian error bound condition lemma, discrete}, this implies the existence of a constant $C_{0, F}$ such that for all $k\geq k_0$
	\begin{equation*}
	   t_{k}^{\eta} F_{k} \leq C_{E} -  t_{k}^{\eta} \varepsilon_{k - 1} H_{k} \leq C_{E} +  C_{0, F} t_{k}^{\eta - \delta}  F_{k}^{\frac{1}{\rho}}.
	\end{equation*}
	Since \eqref{eq: pre rates for Fk, case 0<delta<=1, lemma, discrete} holds, we can apply Lemma \ref{lem: to improve the rate for Fk, discrete}, which yields $F_{k} = \mathcal{O} ( k^{-\eta}  )$ as $k\to+\infty$, that is, Point \ref{item:general_descent_lemma_0_F}. With this result, using \eqref{eq: Hölderian error bound condition lemma, discrete}, for some constant $C_{H} \geq 0$ it holds for all $k\geq k_0$ 
	\begin{equation}
		- t_{k}^{\eta} \varepsilon_{k - 1} H_k \leq t_{k}^{\eta} \varepsilon_{k - 1} C_{3} F_{k}^{\frac{1}{\rho}} \leq  C_H  k^{\eta - \delta - \frac{\eta}{\rho}} = C_H  k^{-(\delta - \frac{\eta}{\rho^\star})}.
	\end{equation}
	This yields Point \ref{item:general_descent_lemma_0_lower_bound}. Since $\delta = \eta\hat \delta> \frac{\eta}{\rho^\star}$, the right-hand side of the previous inequality vanishes as $k\to +\infty$, which, since $F_k, V_k\geq0$, implies that $(E_{k})_{k \geq 1}$ is lower bounded. As $(E_{k + 1} - E_{k})_{k\in\N}$ is upper bounded by a summable sequence, we obtain the existence of $\lim_{k\to +\infty} E_{k}$ and Point \ref{item:general_descent_lemma_0_lim}. 
	
Let us now conclude with the summability statements. In both cases, we have shown that the limit $\lim_{k\to +\infty} E_{k}$ exists. In particular, $(E_k)_{k \geq 1}$ is bounded. Plugging this statement back into \eqref{eq: inequality 1 for E(k+1)-Ek, lemma, discrete} and taking into account that $F_{k}, g_{k} \geq 0$ for all $k \geq k_0$,  gives 
	\begin{equation}\label{eq:proof_lemma_unifying_2}
		\sum_{k = 1}^{+\infty} k^{\eta - 1} F_{k} < +\infty \quad \text{and} \quad \sum_{k = 1}^{+\infty} g_{k} < +\infty. 
	\end{equation}
	Furthermore, relying once again on \eqref{eq: term with Fk, Hk, Thetak is upper bounded by a summable sequence, lemma, discrete}, from \eqref{eq: inequality 2 for E(k+1)-Ek, lemma, discrete} and \eqref{eq: inequality 3 for E(k+1)-Ek, lemma, discrete} we may write, for summable, nonnegative sequences $(C_{1, H}^{k})_{k \geq 1}$, $(C_{3, H}^{k})_{k \geq 1}$ and constants $C_{2, H}, C_{4, H} > 0$, for all $k \geq k_0$
	\begin{equation}\label{eq:proof_lemma_unifying_1}
		\begin{aligned}
			\zeta_{k, \delta} H_{k} &\leq C_{1, H}^{k} - C_{2, H}(E_{k + 1} - E_{k}) \leq C_{1, H}^{k} + C_{2, H} | E_{k + 1} - E_{k}|, \\
			-\zeta_{k, \delta} H_{k} &\leq C_{3, H}^{k} - C_{4, H} (E_{k + 1} - E_{k}) \leq C_{3, H}^{k} + C_{4, H} | E_{k + 1} - E_{k}|. 
		\end{aligned}
	\end{equation}
	Summing up these two inequalities and denoting by $C_{5, H}:=C_{2, H}+ C_{4, H}  >0 $, we obtain for all $k \geq k_0$
	\begin{equation*}
			\begin{aligned}
			& \sum_{\ell=k_0}^k \bigg| C_{1, H}^{\ell} + C_{3, H}^\ell - (C_{2, H} + C_{4, H})(E_{\ell + 1} - E_{\ell}) \bigg|\\
			= & \sum_{\ell=k_0}^k \bigg( C_{1, H}^{\ell} + C_{3, H}^\ell - C_{5, H} (E_{\ell + 1} - E_{\ell}) \bigg) = \sum_{\ell=k_0}^k \bigg( C_{1, H}^{\ell} + C_{3, H}^\ell \bigg) -C_{5, H} (E_{\ell + 1} - E_{k_0}).
			\end{aligned}
	\end{equation*}
	Since the right-hand side has a limit as $k\to +\infty$, and $(C_{1, H}^{k})_{k \geq 1}$, and $(C_{3, H}^k)_{k \geq 1}$ are summable and $C_{5, H} > 0$, we obtain that  $(|E_{k + 1} - E_{k}|)_{k \geq 1}$ is summable as well. Therefore, from \eqref{eq:proof_lemma_unifying_1} and \eqref{eq:general descent lemma for discrete case, assumption on theta} we get
	\begin{equation*}
	\sum_{k = 1}^{+\infty} \zeta_{k, \delta} | H_{k}| < +\infty, \quad \text{thus} \quad \sum_{k = 1}^{+\infty} k^{\eta - 1} \varepsilon_{k - 1} | H_{k}| < +\infty.
	\end{equation*}
	This, together with \eqref{eq:proof_lemma_unifying_2}, concludes the summability statements, i.e., Point \ref{item:general_descent_lemma_1_sum} of the case $\hat \delta=1$ and Point \ref{item:general_descent_lemma_0_sum} of the case $\frac{1}{\rho^\star}<\hat \delta<1$ and hence finishes the proof. 
\end{proof}

\section{Continuous-time analysis}\label{sec:continuous_time}

Our convergence analysis of Algorithm~\ref{alg:first_order} and Algorithm~\ref{alg:second_order} is closely guided by the analysis of the continuous-time systems~\eqref{eq:intro_tichonov_flow} and~\eqref{eq:intro_su_boyd_candes_tichonov}, even if $\hat h=\hat f = 0$. This analysis not only emerges as a foundational result with intrinsic theoretical interest, but also as an excellent assistant to derive convergence results for the corresponding numerical algorithms. In this section, we briefly illustrate and state the main results that can be obtained for \eqref{eq:intro_tichonov_flow} and \eqref{eq:intro_su_boyd_candes_tichonov}, with $\varepsilon: [t_0, +\infty)\to \R_+$ defined by:
\begin{equation}\label{eq:def_epsilon_continuous_time}
	\varepsilon(t) := \frac{c}{t^{\delta}}, \quad \text{for} \ c, \delta > 0, \ \text{and} \ t \geq t_0.
\end{equation}
Observe first that Assumption \ref{ass:attouch_czarnecki} with $\eta \in \{1, 2\}$, originally introduced by Attouch and Czarnecki in \cite{ac10} in the context of continuous-time systems such as \eqref{eq:intro_tichonov_flow}, now reads as: For any solution $x^*$ to the bilevel problem \eqref{eq:bilevel_problem} and all $p \in N_{\argmin f}(x^*)$ it holds
\begin{equation}\label{eq:ass_ac_continuous}
	\int_{t_0}^\infty t^{\eta - 1}\bigg((f- \min f)^*(\varepsilon(t) p) - \sigma_{\argmin f}(\varepsilon(t) p)\bigg) dt < +\infty.
\end{equation}
Furthermore, all auxiliary results including Lemma \ref{lem: sums}, Lemma \ref{lem: for the existence of lim|xk-x*|}, Lemma \ref{lem: discrete Opial}, Lemma \ref{lem: to improve the rate for Fk, discrete}, and Lemma \ref{lem: general descent lemma for discrete case} admit continuous-time counterparts. Specifically, Lemma \ref{lem: sums} can be stated with integrals instead of sums. The continuous-time counterpart of Lemma \ref{lem: for the existence of lim|xk-x*|} can be found, e.g., in \cite[Lemma A.4]{fastOGDA}. The continuous-time version of Opial's Lemma (Lemma \ref{lem: discrete Opial}) is well known, see, e.g., \cite[Lemma A.2]{fastOGDA}. Lemma \ref{lem: to improve the rate for Fk, discrete} can now be stated as follows: For a function $F:[t_0, +\infty)\to \R$, if
\begin{equation}
	F(t)=\mathcal{O}(t^{-\delta}) \quad \text{as}\  t\to+\infty, \quad \text{and} \quad F(t)\leq C_{E} t^{-\eta} + C_{0, F} t^{- \delta} F(t)^{\frac{1}{\rho}} \quad \text{for all} \ t \geq t_0, 
\end{equation}
then $F(t)=\mathcal{O}(t^{-\eta})$ as $t\to +\infty$. This fact can be shown using the same argument of Lemma \ref{lem: to improve the rate for Fk, discrete}. Eventually, a unifying result similar to Lemma \ref{lem: general descent lemma for discrete case} can be established in continuous time by replacing differences of consecutive terms in a sequence by derivatives and sums by integrals. More precisely, we have:
\begin{lem}\label{lem: general descent lemma for continuous case}
	Let $\hat \delta >0$, $\eta \in \{1, 2\}$, and set $\delta := \eta \hat \delta$. Consider a function $E:[t_0, +\infty)\to\R$ defined as 
	$$E(t) = t^{\eta} \Big(F(t) + \varepsilon(t) H(t) \Big) + V(t) \quad \text{for all} \ t \geq t_0,$$ with  $t \mapsto \varepsilon(t)$ as in \eqref{eq:def_epsilon_continuous_time} and $F(t) \geq 0$, $H(t) \geq -H_{*}$, where $H_* \geq 0$, and $V(t) \geq 0$ for all $t \geq t_0$. Suppose that 
	\begin{equation}\label{eq: general inequality for E(k+1)-E(k), lemma, continuous_appendix}
	 \dot E(t) + \zeta (t) H(t) + g(t) \leq -C_{2} t^{\eta - 1} F(t) \quad \text{for all} \ t \geq t_0,
	\end{equation}
	where $g : [t_0, +\infty) \to \R_+$, $C_2 > 0$, and $\zeta : [t_0, +\infty) \to \R$ satisfies, for $C_{1, \zeta} > 0$, $C_{2, \zeta} \geq 0$,
	\begin{equation}\label{eq:general descent lemma for continuous case, assumption on theta}
		C_{1, \zeta} t^{\eta-1} \varepsilon(t) \leq \zeta (t) \leq C_{2, \zeta} t^{\eta - 1} \varepsilon(t) \quad \text{for all} \ t \geq t_0.
	\end{equation}
	Consider the following conditions, which will only be assumed if we explicitly say so. For all $t\geq t_0$: 
	\begin{align*}
		(\labterm{eq: Hölderian error bound condition lemma, continuous}{C1}) & \quad -H(t) \leq C_{3} F(t)^{\frac{1}{\rho}}, \ \text{for some constant $C_{3} \geq 0$ and $1 < \rho \leq 2$. Set $\rho^\star$ as in \eqref{eq:def_rho_star}}, \\
		(\labterm{eq: Attouch-Czarnecki condition lemma, continuous}{C2}) & \quad \text{For each $r > 0$,} \ -t^{\eta - 1} \big( F(t) + r \varepsilon(t) H(t) \big) \leq I^{r}(t), \ \text{where $I^{r} \geq 0$ and $I^r \in \mathbb{L}^{1}[t_0, +\infty)$}.
	\end{align*}
	Then, following statements are true:
	\begin{itemize}
		\item Case $\hat \delta > 1$: It holds $F(t)=\mathcal{O} (t^{-\eta} )$ as $t \to +\infty$;
		\item Case $\hat \delta = 1$: If Assumption \eqref{eq: Attouch-Czarnecki condition lemma, continuous} holds and $C_{2} > 0$ in \eqref{eq: general inequality for E(k+1)-E(k), lemma, continuous_appendix}, then
		\begin{enumerate}
			\item\label{item:continuous_general_descent_lemma_1_lim} The limit $\lim_{t\to +\infty} E(t)$ exists;
			\item\label{item:continuous_general_descent_lemma_1_F} $F(t) = \mathcal{O} ( t^{-\eta} )$ as $t\to +\infty$;
			\item\label{item:continuous_general_descent_lemma_1_sum} The functions $t \mapsto t^{\eta - 1} F(t)$, $t\mapsto t^{\eta - 1}\varepsilon(t) H(t)$ and $t\mapsto g(t)$ are integrable.
		\end{enumerate}
		\item Case $\frac{1}{\rho^\star} < \hat \delta < 1$: If Assumptions \eqref{eq: Hölderian error bound condition lemma, continuous} and \eqref{eq: Attouch-Czarnecki condition lemma, continuous} hold and $C_{2} > 0$ in \eqref{eq: general inequality for E(k+1)-E(k), lemma, continuous_appendix}, then
		\begin{enumerate}
			\item\label{item:continuous_general_descent_lemma_0_F} $F(t) = \mathcal{O}( t^{-\eta} )$ as $t\to+\infty$;
			\item\label{item:continuous_general_descent_lemma_0_lower_bound} It holds $-C_H t^{-\delta + \frac{\eta}{\rho^\star}}\leq t^{\eta} \varepsilon(t) H(t)$ for some $C_H\geq 0$ and all $t \geq t_0$;
			\item\label{item:continuous_general_descent_lemma_0_lim} The limit $\lim_{t\to +\infty} E(t)$ exists;
			\item\label{item:continuous_general_descent_lemma_0_sum} The functions $ \mapsto t^{\eta - 1} F(t)$, $t \mapsto t^{\eta - 1}\varepsilon(t) H(t)$ and $t \mapsto g(t)$ are integrable.
		\end{enumerate} 
	\end{itemize}
\end{lem}

\subsection{First-order system}\label{sec:continuous_time_first_order}

Let us start with the first-order system \eqref{eq:intro_tichonov_flow} with starting point $x(t_0)=x_0\in \cH$. As anticipated in Remark \ref{rem:continuous_time_first_order}, we consider the Lyapunov energy:
\begin{equation*}
	E^\lambda(t):=t(\Psi_t(x(t))-\Psi_t(x^*)) + \frac{\lambda}{2}\|x(t)-x^*\|^2, \quad \text{for} \ t\geq t_0,
\end{equation*}
where $\Psi_{t}(x):=f(x) + \varepsilon(t) h(x)$ is the regularized function \eqref{eq:regularized}, $x^{*}$ is a solution to the bilevel problem \eqref{eq:bilevel_problem}, and $\lambda > 1$. The analogous statement of Lemma \ref{lem:dissipation_first_order} can be easily obtained by taking the time derivative of $E_{\lambda}$. We get for all $t \geq t_0$
\begin{equation*}
	\begin{aligned}
	\Dot{E}_{\lambda}(t) = &\: \Psi_t(x(t)) - \Psi_t(x^{*}) + t\Bigl( \Dot{\varepsilon}(t) (h(x(t)) - h(x^{*})) + \bigl\langle \Dot{x}(t), \nabla f(x(t)) + \varepsilon(t) \nabla h(x(t))\bigr\rangle\Bigr)  \\
	&+ \lambda \bigl\langle x(t) - x^{*}, \Dot{x}(t)\bigr\rangle  \\
	\leq &\: \Psi_t(x(t)) - \Psi_t(x^{*}) - t \bigl\| \Dot{x}(t)\bigr\|^{2} + t\Dot{\varepsilon}(t) (h(x(t)) - h(x^{*})) - \lambda (\Psi_t(x(t)) - \Psi_t(x^{*}))  \\
	= &\: - (\lambda - 1) (f(x(t)) - f(x^{*})) - t\bigl\| \Dot{x}(t)\bigr\|^{2} +  \bigl( t\Dot{\varepsilon}(t) - (\lambda - 1) \varepsilon(t)\bigr) (h(x(t)) - h(x^{*})). 
	\end{aligned}
\end{equation*}
Rearranging the terms and dropping nonnegative ones, for all $t \geq t_0$ we arrive at
\begin{equation}
	\Dot{E}_{\lambda}(t) + \zeta(t) (h(x(t)) - h(x^{*})) \leq - (\lambda - 1) (f(x(t)) - f(x^{*})), \ \mbox{where} \ \zeta(t):=(\lambda - 1) \varepsilon(t) - t\Dot{\varepsilon}(t).
\end{equation}
Applying Lemma \ref{lem: general descent lemma for continuous case} we obtain the following: for case i) $\delta >1$, that $f(x(t)) - \min f = \mathcal{O}(t^{-1})$ as $t\to+\infty$; for both cases ii) $\delta=1$ with \eqref{eq:ass_ac_continuous} and $\eta=1$, or iii) $\delta \in (\frac{1}{\rho^{\star}}, 1)$ with Assumption \ref{ass:holderian_growth}, that for all $\lambda \geq 1$,
\begin{enumerate}
	\item The limit $\lim_{t \to+\infty}E^\lambda(t)$ exists;
	\item $f(x(t)) - \min f = \mathcal{O}(t^{-1})$ as $t\to+\infty$;
        \item In case iii), the following lower bound holds $t\varepsilon(t) (h(x(t) - h(x^*)) \geq - C_h t^{-(\delta - \frac{1}{\rho^\star})}$ for some $C_h\geq 0$ and all $t \geq t_0$;
	\item The following integrability statements hold\begin{equation}\label{eq:integrability_continuous_first_order}
		\int_{t_0}^\infty (f(x(t) - f(x^*))dt <+\infty, \quad \text{and} \quad \int_{t_0}^\infty t^{-\delta}|h(x(t)) - h(x^*)|dt <+\infty.
	\end{equation} 
\end{enumerate}
By taking $1<\lambda_1<\lambda_2$ and considering the difference $E_{\lambda_2}(t) - E_{\lambda_1}(t)$, we deduce that $\varphi(t):= \frac{1}{2}\|x(t) - x^*\|^2$ admits a limit as $t\to+\infty$ for all solutions $x^*$ to the bilevel problem \eqref{eq:bilevel_problem}, and that 
\begin{equation}\label{eq:limit_of_q_continuous_first_order}
	\lim_{t\to +\infty} q(t):= t(f(x(t)) - f(x^*)) + t\varepsilon(t)(h(x(t)) - h(x^*)) \quad \text{exists}.
\end{equation}
From \eqref{eq:limit_of_q_continuous_first_order} and \eqref{eq:integrability_continuous_first_order} we deduce, arguing as in \eqref{eq:proof_thm_first_order_2}, that $\lim_{t \to+\infty} q(t) = 0$. Proceeding as in \eqref{eq:proof_thm_first_order_3} we deduce that $h(x(t)) - h(x^*)$ vanishes as $t\to+\infty$, and, therefore, that all weak sequential cluster points of $t\mapsto x(t)$ as $t\to+\infty$ lie in the solution set of the bilevel problem \eqref{eq:bilevel_problem}. Opial's Lemma in continuous time yields weak convergence of the full trajectory to some solution to the bilevel problem \eqref{eq:bilevel_problem}. To conclude, using a similar argument as in \eqref{eq:proof_thm_first_order_4_1} and \eqref{eq:proof_thm_first_order_4_2} yields the little-$o$ rates. To summarize, we can show the following convergence result for trajectories generated by \eqref{eq:intro_tichonov_flow}.
\begin{thm}\label{thm: first order system_continuous}
Let $x :[t_{0}, +\infty) \to \mathcal{H}$ be a solution trajectory to \eqref{eq:intro_tichonov_flow}. Then, the following statements are true:
\begin{itemize}
	\item Case $\delta > 1$: It holds $f(x(t)) - \min f=\mathcal{O}(t^{-1})$ as $t\to+\infty$.

\item[] \hspace*{-\leftmargin} Suppose Assumption \ref{ass:qualification} holds.

	\item Case $\delta = 1$: If $f$ satisfies \eqref{eq:ass_ac_continuous} with $\eta = 1$, it holds
	\[
	f(x(t)) - \min f = o(t^{-1}) \quad \text{and} \quad h(x(t)) \to \min\nolimits_{\argmin f} h \quad \text{as} \ t\to +\infty.
	\]
Furthermore, $t\mapsto x(t)$ converges weakly towards a solution to the bilevel problem \eqref{eq:bilevel_problem};
	\item Case $ \frac{1}{\rho^\star} < \delta < 1$: If $f$ satisfies Assumption \ref{ass:holderian_growth} with exponent $\rho$, it holds
	\[
	f(x(t)) - \min f = o(t^{-1}) \quad \text{and} \quad |h(x(t)) - \min\nolimits_{\argmin f} h| = o(t^{-1+\delta}) \quad \text{as} \ t \to +\infty,
	\]
Furthermore, $t\mapsto x(t)$ converges weakly towards a solution to the bilevel problem \eqref{eq:bilevel_problem}.
\end{itemize}
\end{thm}

\subsection{Second-order system}

We now consider the second-order system \eqref{eq:intro_su_boyd_candes_tichonov} with starting position $x(t_0)=x_0\in \cH$ and starting velocity $\dot x(t_0)=v$. As anticipated in Remark \ref{rem:continuous_time_second_order}, the continuous-time counterpart of the Lyapunov energy \eqref{eq:lyapunov_second_order} is defined for  $\alpha >3$, $\lambda \in (2, \alpha - 1)$, and $t\geq t_{0}$ as 
\begin{equation}
	E_{\lambda}(t) := t^{2} (\Psi_t(x(t)) - \Psi_t(x^{*})) + \frac{1}{2} \|\lambda(x(t) - x^{*}) + t\Dot{x}(t) \|^{2} + \frac{\lambda(\alpha - 1 - \lambda)}{2} \| x(t) - x^{*}\|^{2},
\end{equation}
where, once again, $\Psi_t(x) := f(x) + \varepsilon(t) h(x)$ and $x^{*}$ is a solution to the bilevel problem \eqref{eq:bilevel_problem}. To derive the analogous result to Lemma \ref{lem:dissipation_second_order}, we denote by $v_{\lambda}(t):=\lambda(x(t) - x^{*}) + t\Dot{x}(t)$ and take the time derivative of $E^\lambda$ to get for all $t>t_0$
\begin{equation*}
	\begin{aligned}
	\Dot{E}_{\lambda}(t) = &\: 2t (\Psi_t(x(t)) - \Psi_t(x^{*})) + t^{2} \Bigl( \Dot{\varepsilon}(t)(h(x(t)) - h(x^{*})) + \bigl\langle \Dot{x}(t), \nabla f(x(t)) + \varepsilon(t) \nabla h(x(t))\bigr\rangle\Bigr) \\
	&+ \bigl\langle v_{\lambda}(t), \Dot{v}_{\lambda}(t)\bigr\rangle + \lambda(\alpha - 1 - \lambda) \bigl\langle x(t) - x^{*}, \Dot{x}(t)\bigr\rangle \\
	= &\: 2t (\Psi_t(x(t)) - \Psi_t(x^{*})) + t^{2} \Bigl( \Dot{\varepsilon}(t)(h(x(t)) - h(x^{*})) + \bigl\langle \Dot{x}(t), \nabla f(x(t)) + \varepsilon(t) \nabla h(x(t))\bigr\rangle\Bigr) \\
	& -  \lambda(\alpha - 1 - \lambda) \bigl\langle x(t) - x^{*}, \Dot{x}(t)\bigr\rangle -\lambda t \bigl\langle x(t) - x^{*}, \nabla f(x(t)) + \varepsilon(t) \nabla h(x(t))\bigr\rangle  \\
	&- (\alpha - 1 - \lambda) t \bigl\| \Dot{x}(t)\bigr\|^{2} - t^{2} \bigl\langle \Dot{x}(t), \nabla f(x(t)) + \varepsilon(t) \nabla h(x(t))\bigr\rangle + \lambda(\alpha - 1 - \lambda) \bigl\langle x(t) - x^{*}, \Dot{x}(t)\bigr\rangle  \\
	= &\: 2t(\Psi_t(x(t)) - \Psi_t(x^{*})) + t^{2} \Dot{\varepsilon}(t) (h(x(t)) - h(x^{*}))  - (\alpha - 1 - \lambda) t \bigl\| \Dot{x}(t)\bigr\|^{2}  \\
	&- \lambda t \bigl\langle x(t) - x^{*}, \nabla f(x(t)) + \varepsilon(t) \nabla h(x(t))\bigr\rangle  \\
	\leq &\: -t(\lambda-2)(\Psi_t(x(t)) - \Psi_t(x^{*})) + t^{2} \Dot{\varepsilon}(t) (h(x(t)) - h(x^{*})) + (\lambda + 1 - \alpha) t \bigl\| \Dot{x}(t)\bigr\|^{2}.
	\end{aligned}
\end{equation*}
Rearranging suitably and using the definition of $\Psi_t$, we obtain:
\begin{equation}\label{eq:dissipation_second_order_continuous}
\Dot{E}_{\lambda}(t) + \zeta(t) (h(x(t)) - h(x^{*})) + (\alpha - 1 - \lambda) t \bigl\| \Dot{x}(t)\bigr\|^{2}\leq -t(\lambda - 2) (f(x(t)) - f(x^{*})),
\end{equation}
where $\zeta(t):= t(\lambda - 2)\varepsilon(t) - t^{2} \Dot{\varepsilon}(t)$ for all $t \geq t_0$. We have obtained in \eqref{eq:dissipation_second_order_continuous} the continuous-time counterpart to the statement of Lemma \ref{lem:dissipation_second_order}. Using once again the continuous-time counterpart of Lemma \ref{lem: general descent lemma for discrete case} (i.e., Lemma \ref{lem: general descent lemma for continuous case}) we obtain the following: for case i) $\delta > 2$, that $f(x(t)) - \min f = \mathcal{O}(t^{-2})$ as $t\to+\infty$; for both cases ii) $\delta=2$ with \eqref{eq:ass_ac_continuous} and $\eta =2$, and iii) $\delta \in (\frac{2}{\rho^{\star}}, 2)$, with Assumption \ref{ass:holderian_growth}, that for all $\lambda \in (2,\alpha-1)$, 
\begin{enumerate}
	\item The limit $\lim_{t \to+\infty}E^\lambda(t)$ exists;
	\item $f(x(t)) - \min f = \mathcal{O}(t^{-2})$ as $t\to+\infty$;
        \item In case iii), the following lower bound holds $t^2\varepsilon(t) (h(x(t) - h(x^*)) \geq - C_h t^{-(\delta - \frac{2}{\rho^\star})}$ for some $C_h\geq 0$ and all $t \geq t_0$;
	\item The following integrability statements hold
	\begin{equation}
		\int_{t_0}^\infty t(f(x(t)) - f(x^*))dt <+\infty, \quad \text{and} \quad \int_{t_0}^\infty t^{1-\delta}|h(x(t)) - h(x^*)|dt <+\infty.
	\end{equation} 
\end{enumerate}
Arguing similarly as in the proof of Theorem \ref{thm:second_order_main} we obtain the following convergence result.
\begin{thm}\label{thm: second order system}
Let $x : [t_{0}, +\infty) \to \mathcal{H}$ be a solution trajectory to \eqref{eq:intro_su_boyd_candes_tichonov}. Then, the following statements are true:
	\begin{itemize}
		\item Case $\delta > 2$: It holds $f(x(t)) - \min f = \mathcal{O}(t^{-2})$ as $t\to+\infty$.
        
 \item[] \hspace*{-\leftmargin} Suppose Assumption \ref{ass:qualification} holds.       
 
		\item Case $\delta = 2$: If $f$ satisfies \eqref{eq:ass_ac_continuous} with $\eta = 2$, it holds
		\[
		f(x(t)) - \min f = o(t^{-2}) \quad \text{and} \quad h(x(t)) \to \min\nolimits_{\argmin f} h \quad \text{as} t \ \to +\infty. 
		\]
Furtheremore, $t\mapsto x(t)$ converges weakly towards a solution to the bilevel problem \eqref{eq:bilevel_problem}.

		\item Case $\frac{2}{\rho^\star} < \delta < 2$: If $f$ satisfies Assumption \ref{ass:holderian_growth} with exponent $\rho$, it holds
		\[
		f(x(t)) - \min f = o(t^{-2}) \quad \text{and} \quad |h(x(t)) - \min\nolimits_{\argmin f} h| = o(t^{-2+\delta}) \quad \text{as} \ t\to +\infty.
		\]
Furthermore, $t\mapsto x(t)$ converges weakly towards a solution to the bilevel problem \eqref{eq:bilevel_problem}.
	\end{itemize}
\end{thm}

\printbibliography

\end{document}